\documentclass[oneside, 10pt, article, draft]{memoir}

\usepackage{amssymb}
\usepackage{amsmath}
\usepackage{amsfonts}
\usepackage{amsthm}

\usepackage[english]{babel}

\usepackage{graphicx}

\usepackage[plain, english]{fancyref}
\usepackage{hyperref}
\hypersetup{
     colorlinks   = true,
     citecolor    = blue,
     linkcolor    = blue,
     final
}

\usepackage{bbm}
\usepackage{mathrsfs}
\usepackage{stmaryrd}

\usepackage[all, line]{xy}

\usepackage{ifthen}

\usepackage{xspace}

\usepackage{calc}
\usepackage{color}

\usepackage{fixme}

\usepackage{cancel} 

\usepackage{verbatim} %used for comments
\usepackage{xparse}
\usepackage{environ}

\usepackage{scalerel}

\def\clap#1{\hbox to 0pt{\hss#1\hss}}

\def\mathrlap{\mathpalette\mathrlapinternal}

\def\mathrlapinternal#1#2{%
           \rlap{$\mathsurround=0pt#1{#2}$}}

\makeatletter
\def\blfootnote{\xdef\@thefnmark{}\@footnotetext}
\makeatother

\newcommand{\mathset}[1]{\mathbbm{#1}}

\newcommand{\setN}{\mathset{N}}
\newcommand{\setZ}{\mathset{Z}}

\newcommand{\setR}{\mathset{R}}
\newcommand{\setC}{\mathset{C}}

\newcommand{\trdot}{\smash{\cdot}}
\newcommand{\lcc}{{\nabla^{g}}}
\newcommand{\bwedge}{\raise1pt\hbox{\ensuremath{\bigwedge}}}

\renewcommand{\d}{\partial}

\newcommand{\mc}{c}

\newcommand{\ccint}{\lambda}

\newcommand{\qs}{H^{(k)}}
\newcommand{\pqs}{\mathcal{H}^{(k)}}

\newcommand{\qb}{\hat H^{(k)}}
\newcommand{\pqb}{\mathcal{\hat H}^{(k)}}

\newcommand{\trivcon}{\nabla^{\:\!\raise1.0pt\hbox{$\scriptscriptstyle T$}}}

\newcommand{\sssetC}{{\scriptscriptstyle \setC}}

\newcommand{\g}{\mathfrak{g}}
\newcommand{\gc}{\mathfrak{g}_{\sssetC}}

\renewcommand{\phi}{\varphi}
\renewcommand{\epsilon}{\varepsilon}
\renewcommand{\otimes}{\varotimes}

\newcommand{\sotimes}{\mathbin{\!\!\raise1.5pt\hbox{
      $\scriptscriptstyle\otimes$}}}

\newcommand{\boldnabla}{\mbox{\boldmath$\nabla$}}
 
\newcommand{\com}[2]{\big [ #1, #2 \big ]}

\newsavebox\CBox
\newcommand\hcancel[2][0.2pt]{%
  \ifmmode\sbox\CBox{$#2$}\else\sbox\CBox{#2}\fi%
  \makebox[0pt][l]{\usebox\CBox}%  
  \rule[0.3\ht\CBox-#1/2]{\wd\CBox}{#1}}

\newcommand{\hcc}{F_{\raisebox{-1pt}{\mbox{\boldmath$\hspace{-0.5pt}
        {\scriptstyle \nabla}$}}}}     
\newcommand{\hcct}{F_{\raisebox{0pt}{\mbox{\boldmath$\hspace{-0.5pt}
        {\scriptstyle \tilde \nabla}$}}}}

\newcommand\scale[2]{{#2}}  %ignore scaling for faster compilation

\DeclareMathOperator{\im}{Im}       
       
\DeclareMathOperator{\Id}{Id}
        
\DeclareMathOperator{\SU}{SU}        
\DeclareMathOperator{\U}{U}        
\DeclareMathOperator{\SL}{SL}
\DeclareMathOperator{\PSL}{PSL}

\DeclareMathOperator{\Tr}{Tr}
\DeclareMathOperator{\Hom}{Hom}     
\DeclareMathOperator{\End}{End}

\DeclareMathOperator{\re}{Re}

\DeclareMathOperator{\Sym}{\mathcal{S}}

\theoremstyle{plain}
\newtheorem{theorem}{Theorem}[chapter]
\newtheorem{lemma}[theorem]{Lemma}
\newtheorem{proposition}[theorem]{Proposition}

\newtheorem{definition}[theorem]{Definition}

\theoremstyle{definition}

\makeatletter       % Gør at vi kan bruge @ i kommandonavne.

\let\@xp\expandafter % Vi skal bruge \expandafter en del, så en
\newcommand\DefineFancyrefPrefix[2]{%
  \@namedef{fancyref#1labelprefix}{#1}%
  \@namedef{Fref#1name}{#2}%
  \@namedef{fref#1name}{\MakeLowerCase{\@nameuse{Fref#1name}}}%
  \def\@style{vario}%
  \@xp\@xp\@xp\frefformat\@xp\@xp\@xp\@style\@xp\csname
  fancyref#1labelprefix\endcsname
  {%
    \@nameuse{fref#1name}\fancyrefdefaultspacing##1##3%
  }%
  \@xp\@xp\@xp\Frefformat\@xp\@xp\@xp\@style\@xp\csname
  fancyref#1labelprefix\endcsname
  {%
    \@nameuse{Fref#1name}\fancyrefdefaultspacing##1##3%
  }%
  \def\@style{plain}%
  \@xp\@xp\@xp\frefformat\@xp\@xp\@xp\@style\@xp\csname
  fancyref#1labelprefix\endcsname
  {%
    \@nameuse{fref#1name}\fancyrefdefaultspacing##1%
  }%
  \@xp\@xp\@xp\Frefformat\@xp\@xp\@xp\@style\@xp\csname
  fancyref#1labelprefix\endcsname
  {%
    \@nameuse{Fref#1name}\fancyrefdefaultspacing##1%
  }%
}
\makeatother

\DefineFancyrefPrefix{ax}{Axiom}
\DefineFancyrefPrefix{cor}{Corollary}
\DefineFancyrefPrefix{lem}{Lemma}
\DefineFancyrefPrefix{prop}{Proposition}
\DefineFancyrefPrefix{thm}{Theorem}
\DefineFancyrefPrefix{def}{Definition}
\DefineFancyrefPrefix{rem}{Remark}
\DefineFancyrefPrefix{ex}{Example}
\DefineFancyrefPrefix{cha}{Chapter}

\let\amsthmproof\proof
\let\amsthmendproof\endproof

\RenewDocumentEnvironment{proof}{s}
 {\IfBooleanTF{#1}
   {\ignoreproof}
   {\amsthmproof}}
 {\IfBooleanTF{#1}
   {\endignoreproof}
   {\amsthmendproof}}

\NewEnviron{ignoreproof}{\begin{proof} The proof is hidden. \end{proof}}

\makepagestyle{niels}

\makeoddfoot{niels}{}{\thepage}{}
\setlrmarginsandblock{3.9cm}{3.9cm}{*} % left and right 
\checkandfixthelayout[nearest]

\chapterstyle{article}
\begin{document}

\title{The Hitchin-Witten Connection and \\ Complex Quantum
  Chern-Simons Theory} 

\author{J{\o}rgen Ellegaard Andersen and Niels Leth Gammelgaard}

\blfootnote{Supported in part by the center of excellence grant ``Center for quantum geometry of Moduli Spaces" from the Danish National Research Foundation)}

\maketitle

\begin{abstract}
  We give a direct calculation of the curvature of the Hitchin
  connection, in geometric quantization on a symplectic manifold,
  using only differential geometric techniques. In particular, we
  establish that the curvature acts as a first-order operator on the
  quantum spaces. Projective flatness follows if the K\"ahler
  structures do not admit holomorphic vector fields. Following Witten,
  we define a complex variant of the Hitchin connection on the bundle
  of prequantum spaces. The curvature is essentially unchanged, so
  projective flatness holds in the same cases. Finally, the results
  are applied to quantum Chern-Simons theory, both for compact and
  complex gauge groups.
\end{abstract}

%\listoffixmes

\chapter{Introduction}

Since their introduction by Atiyah \cite{MR1001453}, Segal
\cite{MR981378} and Witten \cite{MR953828,MR990772}, topological
quantum field theories (TQFTs) have been studied intensely using a
wide range of techniques. The first construction in 2 + 1 dimensions
was given by Reshetikhin and Turaev
\cite{MR1036112,MR1091619,MR1292673} using representation theory of
quantum groups at roots of unity to construct link invariants and in
turn derive invariants of 3-manifolds through surgery and Kirby
calculus. Shortly thereafter, a combinatorial construction was given
by Blanchet, Habegger, Masbaum and Vogel \cite{MR1191373,MR1362791} in
the language of skein theory. 

A geometric realization was proposed by
Witten \cite{MR953828}, suggesting the use of quantum Chern-Simons
theory or conformal field theory to construct the 2-dimensional
part. The gauge theoretic approach was studied independently by
Axelrod, Della Pietra and Witten \cite{MR1100212} and Hitchin
\cite{MR1065677},
% (see also \cite{MR2928087} and \cite{MR2928088} for a purely
% differential geometric approach),
proving that the quantum spaces arising from geometric quantization of
Chern-Simons theory for compact gauge group are indeed independent of
the conformal structure on the surface, in the sense that they are
identified by parallel transport of a projectively flat connection
over the Teichm\"uller space of the surface. These constructions have
been expressed and generalized in purely differential geometric terms
in \cite{MR2928087} and \cite{MR2928088}, and we shall be mainly
concerned with this description in the present paper. The other
construction proposed by Witten, through conformal field theory, was
provided by Tsuchiya, Ueno and Yamada \cite{MR1048605}, and the link
to the gauge theoretic construction was established by Laszlo
\cite{MR1669720}. 

Only recently has the relation to Reshetikhin and Turaev been fully
demonstrated. In a series of papers
\cite{MR2339577,MR2306213,MR2928086,1110.5027}, the first author of
this paper and Ueno obtain a modular functor, from a twist of the
conformal field theory construction, and identify it with the modular
functor constructed from skein theory, and hence with the original
construction of Reshetikhin and Turaev. This has paved the way for
studying the TQFT through geometric quantization of moduli space and
in particular the application of Toeplitz operator theory, see
e.g. \cite{MR3181488,MR2195137,MR2681692,MR2883417,MR2928087,MR2436739,MR2595923,MR2807846,MR2928090,1408.2499}.
% \cite{A1,A2,A3,AB,A4,A5,A6,AG,AH,1408.2499}.

For non-compact gauge group, the situation is very different. In the
paper \cite{MR1099255}, Witten initiated the study of Chern-Simons
topological quantum field theory for complex gauge groups from a
physical point of view. He proposed that the smooth sections of the
Chern-Simons line bundle over the moduli space of flat connections in
the corresponding compact real form of the group should be the
appropriate pre-Hilbert space of this theory. He reduced the
description to this model space by considering a real polarization,
which we review in section \ref{QCST}, on the space of connections
with values in the complex gauge group. Although this model space
itself does not depend on a choice of complex structure on the
surface, the polarization, and hence the interpretation as the quantum
space, does. Witten argued that the needed infinitesimal change of
polarization, under infinitesimal change of the complex structure on
the surface, can be encoded as a connection in the trivial bundle with
the fixed model space as fiber. Furthermore, Witten provided
infinite-dimensional gauge theory arguments for the projective
flatness of this connection.

In this paper, we review the general differential geometric
construction of the Hitchin connection for a rigid family of complex
structures on a symplectic manifold with vanishing first Betti number
and first Chern class represented essentially by the symplectic
form. Inspired by Witten's considerations, we also consider a certain
one-parameter family of connections on the prequantum spaces. For any
value of the parameter, the resulting connection will be called the
Hitchin-Witten connection. The main feature of both connections is
projective flatness, at least when the K\"ahler structures have few
symmetries, and we shall establish this fact by direct curvature
calculations.

Let us briefly introduce the setting and state the main results. The
basic assumptions and their implications will be explored in greater
detail in the following sections. Consider a symplectic manifold
$(M,\omega)$, with vanishing first Betti number and first Chern class
given by $c_1(M, \omega) = \ccint \big [ \tfrac{\omega}{2\pi} \big]$,
for some integer $\ccint$. Futhermore, let $\mathcal{T}$ be a complex
manifold parametrizing a holomorphic family $J \colon \mathcal{T} \to
C^\infty(M, \End(TM))$ of integrable almost complex structures on $(M,
\omega)$, none of which admit non-constant holomorphic functions on
$M$. The variation of the K\"ahler structure, along the holomorphic
part of a vector field on $V$ on $\mathcal{T}$, is encoded by a
section $G(V) \in C^\infty(M, S^2(T'\!M))$, defined by $V'[J] = G(V)
\trdot \omega$, of the second symmetric power of the holomorphic
tangent bundle, and we will assume that the family $J$ is \emph{rigid}
in the sense that the bivector field $G(V)$ defines a holomorphic
section $G(V) \in H^0_J(M,S^2(T'\!M))$.

By the assumption on the Chern class, the symplectic manifold $(M,
\omega)$ admits a Hermitian line bundle $\mathcal{L}$ with a
compatible connection of curvature $F_\nabla = -i\omega$. For any
$\sigma \in \mathcal{T}$, the space $\qs_\sigma = H^0_\sigma(M,
\mathcal{L}^k)$ of holomorphic setions is the quantum space, at level
$k \in \setN$, arising from geometric quantization using the complex
structure $J_\sigma$. These spaces sit inside the prequantum space
$\pqs = C^\infty(M, \mathcal{L}^k)$ of smooth sections, and as $\sigma
\in \mathcal{T}$ varies, they form a sub-bundle of the trivial bundle
$\pqb = \mathcal{T} \times \pqs$. As proved in \cite{MR2928087}, this
sub-bundle is preserved by the explicitly given connection,
\begin{align*}
  \boldnabla_V = \trivcon_V + \frac{1}{4k+2n}(\Delta_{G(V)} + 2
  \nabla_{G(V) \trdot dF} - 2\ccint V'[F]) + V'[F],
\end{align*}
where $\Delta_{G(V)}$ is a second-order operator with symbol $G(V)$,
and $F\in C^\infty(\mathcal{T} \times M)$ is the Ricci potential,
expressing the relation between the Ricci form $\rho$ and the
symplectic form through $\rho = \ccint \omega + 2i \d \bar \d F$. This
connection will be called the Hitchin connection, and generalizes the
connections studied by \cite{MR1065677} and \cite{MR1100212} in the
setting of quantum Chern-Simons theory for compact gauge group.

Inspired by the work of Witten \cite{MR1099255}, we may also use the
family of K\"ahler structures to define another connection on $\pqb$
by the expression,
\begin{align*}
  \tilde \boldnabla_V \!=\! \trivcon_V \! + \!\frac{1}{2t}
  (\Delta_{G(V)} \! + \!2 \nabla_{G(V) \trdot dF}\! -\! 2 \ccint
  V'[F]) - \frac{1}{2\bar t} (\Delta_{\bar G(V)} \! + \!2 \nabla_{
    \bar G(V) \trdot dF} \!- \! 2\ccint V''[F]) \!+ \! V[F],
\end{align*}
for any complex number $t \in \setC$ with real part equal to $k$. We
shall refer to this as the Hitchin-Witten connection at level
$k$. Of course it also depends on the imaginary part of \nolinebreak[1] $t$,
so for each level $k$ we get family of theories, parametrized by one
real parameter.

Unlike the Hitchin connection, the Hitchin-Witten connection will not
preserve the sub-bundle of quantum spaces. Both connections do,
however, share another meritorious feature, captured by the main
theorem of the paper.

\begin{theorem}
  \label{thm:MT}
  If all complex structures in the family have zero-dimensional
  symmetry group, then both the Hitchin and Hitchin-Witten connections
  are projectively flat.
\end{theorem}

Since the moduli spaces have no holomorphic vector fields, which is
equivalent to the condition stated in the theorem, we get the
following immediate corollary

\begin{theorem}
  \label{thm:MTCS}
  The Hitchin and Hitchin-Witten connections for the moduli spaces of
  flat connections on a closed oriented surface are projectively flat.
\end{theorem}

This result on the Hitchin connection, for the moduli spaces of flat
$\SU(n)$ connections, is due to Hitchin, but his proof uses algebraic
geometric properties of these moduli spaces. As mentioned above,
Witten gave an infinite-dimensional gauge-theoretic argument for this
result on the Hitchin-Witten connection, for the same moduli
spaces. In this paper, we provide a purely differential geometric
finite-dimensional argument for projective flatness, which applies in
the more general setting we have described.

As explained above, the quantum representation of the mapping class
groups from the Reshetikhin-Turaev TQFT for $\SU(n)$ coincides with
the representation obtained from the Hitchin connection. We expect
that the same will hold for the Hitchin-Witten connection and quantum
Chern-Simons theory for the complex gauge group $\SL(n,\setC)$. In the
paper \cite{GenusoneSL2C}, the first author has computed explicitly
the resulting representation of the mapping class group for genus $1$,
the gauge group $\SL(2,\setC)$ and all integer levels $k$.

Only recently has the whole TQFT been rigorously constructed by the
first author of this paper and Kashaev in \cite{MR3227503} and
\cite{1305.4291} for the case of $\PSL(2,\setC)$ and level $k=1$,
using quantum Teichm\"{u}ller theory and the Faddeev quantum
dilogarithm. In the paper \cite{AK_CQCS}, the same authors have
constructed quantum Chern-Simons theory for $\PSL(2,\setC)$ and all
non-negative integer levels $k$ and further understood how it relates
to the geometric quantization of the $\PSL(2,\setC)$-moduli spaces. In
fact, they have proposed a very general scheme which just requires a
Pontryagin self-dual locally compact group, which we expect will lead
to the construction of the $\SL(n,\setC)$ for all non-negative integer
levels $k$. This should be seen in parallel to the developments on indeces \cite{1208.1663,MR3073925}, which should be related to the level $k=0$ theory. In the physics literature, the complex quantum Chern-Simons theory has been discussed from a path integral point of view in a number of
papers \cite{MR3250765,MR3148093,MR2551896,MR3080552,MR2465747,MR2134725,Hik1,Hik2,MR2809462,MR1133274} and latest by
Dimofte \cite{Dimofte3d3d} using the more advanced 3d-3d
correspondence.

\subsection*{Outline}

Let us briefly outline the organization of the paper.  Section
\ref{cha:famil-kahl-struct} discusses general aspects of families of
K\"ahler structures on a fixed symplectic manifold. In particular, a
rather serious holomorphicity condition, called rigidity
(\Fref{def:4}), on the infinitesimal deformation of the K\"ahler
structure will be discussed. We derive a symmetry result
(\Fref{prop:2}) for certain tensor fields associated with such
families, which will prove crucial in the calculation of the curvature
of the Hitchin connection.
% Furthermore, we study the canonical line bundle of the family and
% use the Bianchi identity to derive a formula \eqref{eq:1234} for the
% variation of the Ricci form in terms of the variation of complex
% structure. This formula is applied to derive an important identity
% \eqref{eq:41} satisfied by any family of Ricci potentials.

In section \ref{cha:geom-quant}, we briefly recall the basics of
geometric quantization, described as a two stage process where the
prequantum space $\pqs$ is first constructed as sections of a line
bundle, and the quantum space $\qs$ is then defined as the subspace of
polarized sections with respect to some choice of auxiliary K\"ahler
polarization. One way of understanding the influence of this choice on
quantization is by studying the infinitesimal behaviour of a family of
such polarizations, and this is exactly the approach employed with the
Hitchin connection, which relates the quantum spaces through parallel
transport. Following the initial discussion, we calculate the
commutators of general second-order differential operators acting on
the prequantum spaces. These results will be useful when calculating
the curvature of the Hitchin connection.

The Hitchin connection is the subject of section
\ref{cha:hitchin-connection}, which contains the main results. After
briefly reviewing the differential geometric construction of the
Hitchin connection, we move on to the straightforward but rather
lengthy calculation of its curvature, culminating with the first major
result in \Fref{thm:4}. The fact that the curvature acts as a
differential operator of order at most one is a crucial
point. Building on these computations and inspiration from Witten's
work on quantum Chern-Simons theory for complex gauge group
\cite{MR1099255}, we then consider the Hitchin-Witten connection
defined in \eqref{eq:38} as well as above. The second major result is
the calculation of its curvature in \Fref{thm:1}. The expressions turn
out to be essentially equal to the curvature of the Hitchin connection
in \Fref{thm:4}, and in particular it acts as a differential operator
of order at most one. This property, shared by both connections,
entails projective flatness if, for instance, the family of complex
structures does not admit holomorphic vector fields. This is the
content of \Fref{thm:2}, leading ultimately to \Fref{thm:MT}.

The final section applies the results to quantum Chern-Simons theory,
both for compact and complex gauge groups. The Hitchin connection was
originally studied in this setting, with compact gauge group $\SU(n)$,
by Hitchin \cite{MR1065677} and Axelrod, Della Pietra and
\mbox{Witten \cite{MR1100212}}. Our results provide another proof of 
projective flatness, which still relies on the absence of holomorphic
vector fields, but uses general properties of rigid families of
K\"ahler structures to prove the vanishing of higher-order symbols. As
mentioned, the case of complex gauge group was studied by Witten in
\cite{MR1099255}, where he used a real polarization to reduce the
quantum space from complex to unitary connections. The real
polarization, and hence the reduction, depends on the conformal
structure on the surface, and Witten derived a formula for an analogue
of the Hitchin connection in this model, arriving at exactly the
expression \eqref{eq:38}. We recall the necessary theory and connect
it with the results of previous sections, providing a differential
geometric and purely finite-dimensional construction of a projectively
flat connection, the Hitchin-Witten connection, in Witten's model of
quantum Chern-Simons theory with complex gauge group.

\pagebreak[0]

\chapter{Families of K\"ahler Structures}
\label{cha:famil-kahl-struct}

Before we recall the construction of the Hitchin connection and
calculate its curvature, we will explore the properties of families of
K\"ahler structures on a symplectic manifold. Such families are
central to the notion of a Hitchin connection, and the results
obtained will play a fundamental role in subsequent parts. The
section serves to introduce notation, establish conventions and provide a
number of basic results for later reference. The result is somewhat
lengthy and can be read swiftly on first reading. 

Let $(M, \omega)$ be a symplectic manifold. If $\mathcal{T}$ is a
manifold, we say that a smooth map,
\begin{align*}
  J \colon \mathcal{T} \to C^\infty(M, \End(TM)),
\end{align*}
is a family of K\"ahler structures on $(M, \omega)$ if it defines an
integrable and $\omega$-compatible almost complex structure for every
point $\sigma \in \mathcal{T}$. Smoothness of $J$ means that it
defines a smooth section of the pullback bundle $\pi_M^* \End(TM)$
over $\mathcal{T} \times M$, where $\pi_{M} \colon \mathcal{T} \times
M \to M$ denotes the projection.

For any point $\sigma \in \mathcal{T}$, the almost complex structure
$J_\sigma$ induces a splitting,
\begin{align*}
  TM_\setC = T'\!M_\sigma \oplus T''\!M_\sigma,
\end{align*}
of the complexified tangent bundle of $M$ into the two eigenspaces of
$J_\sigma$, with associated subspace projections $ \pi_\sigma^{1,0}
\colon TM_\setC \to T'\!M_\sigma$ and $\pi^{0,1}_\sigma \colon
TM_\setC \to T''\!M_\sigma$, explicitly given by
\begin{align*}
  \pi^{1,0}_\sigma = \tfrac{1}{2} (\Id - iJ_\sigma) \qquad \text{and}
  \qquad \pi^{0,1}_\sigma = \tfrac{1}{2} (\Id + iJ_\sigma).
\end{align*}
We denote by $X = X'_\sigma + X''_\sigma$ the associated splitting of
a vector field $X$ on $M_\sigma$.  In general, the subscript $\sigma$
indicates dependence on the complex structure, but we shall typically
ommit it when the dependence is obvious and the formula is valid for
any point in $\mathcal{T}$.

The K\"ahler metric associated with the complex structure $J$ is given
by
\begin{align*}
  g = \omega \trdot J,
\end{align*}
where the dot denotes contraction of tensors.  The inverses of $g$ and
$\omega$ are denoted by $\tilde g$ and $\tilde \omega$, respectively,
and they are the unique symmetric, respectively anti-symmetric,
bivector fields satisfying
\begin{align*}
  g \trdot \tilde g = \tilde g \trdot g = \Id \qquad \text{and} \qquad
  \omega \trdot \tilde \omega = \tilde \omega \trdot \omega = \Id.
\end{align*}
As above, a dot will be used to denote contraction of tensors, and the
placement of the tensors relative to each other indicates which
entries to contract. We will, however, also encounter more complicated
expressions, where the entries to be contracted cannot be indicated by
simply placing the tensors next to each other. In such cases, we will
use abstract indices to denote the entries of each tensor. The indices
only name the entries of a tensor, and do not represent a choice of
local coordinates. As usual, subscript indices refer to covariant
entries of a tensor, whereas superscript indices refer to
contravariant entries, and following the Einstein convention, repeated
indices will indicate contraction. If the two contracted indices are
both either subscript or superscript, the K\"ahler metric is used for
contraction. With these conventions, the above identities become
\begin{align*}
  g_{ab} = \omega_{au} J_b^u \qquad \text{and} \qquad g_{au}\tilde
  g^{ub} = \omega_{au} \tilde \omega ^{ub} = \Id_a^b.
\end{align*}
In a few places, we will also need to apply the projections
$\pi^{1,0}$ and $\pi^{0,1}$ to the entries of the tensor. In the index
notation, composition with $\pi^{1,0}$ will be indicated by a prime on
the index, whereas composition with $\pi^{0,1}$ will be indicated by
two primes. As an example we can write
\begin{align*}
  \omega_{ab} = \omega_{a'b''} + \omega_{a'' b'}
\end{align*}
for the symplectic form, which is of type (1,1) with respect to a
compatible complex structure.

Associated with the K\"ahler metric $g$, we have the Levi-Civita
connection $\lcc$, and as usual its curvature is defined by
\begin{align*}
  R(X,Y)Z = \nabla_X \nabla_Y Z - \nabla_Y \nabla_X Z - \nabla_{[X,Y]}
  Z,
\end{align*}
for any vector fields $X, Y$ and $Z$ on $M$. In abstract index
notation, the curvature writes $R_{abc}^d$, and we can use the metric
to lower the upper index and get the curvature tensor
\begin{align*}
  R_{abcd} = R_{abc}^rg_{rd}.
\end{align*}
The Ricci curvature is the symmetric $J$-invariant tensor $r$ defined
by
\begin{align*}
  r(X,Y) = \Tr ( Z \mapsto R(Z,X)Y),
\end{align*}
and its corresponding skew-symmetric two-form is the Ricci form $\rho
= J \trdot r$, which would correspond to
\begin{align*}
  r_{ab} = R_{uab}^u = R_{uabu} = R_{auub} \qquad \text{and} \qquad
  \rho_{ab} = J^u_a r_{ub} = \frac{1}{2} R_{abuv} \tilde \omega^{uv}
\end{align*}
in index notation. Finally, the scalar curvature $s$ is the metric
trace of the Ricci curvature
\begin{align*}
  s = r_{uu} = r_{uv} \tilde g^{uv} = \rho_{uv} \tilde \omega^{vu}.
\end{align*}

If $E$ is a vector bundle over $M$ and $D \in \mathcal{D}(M, E)$ is a
differential operator of order at most $n$, we can assign the
principal symbol $\sigma_P(D) \in C^\infty(M, S^n(TM))$, which is a
symmetric section of the $n$'th tensor power of the tangent bundle. If
the principal symbol vanishes, then $D$ is of order at most $n-1$.  In
general, there is no good notion of lower order symbols of
differential operators, but a connection on $E$ and can be combined
with the Levi-Civita connection on $M$ to define symbols of all
orders. For vector fields $X_1, \ldots, X_n$ on $M$, we consider the
inductively defined differential operator on sections of $E$,
\begin{align}
  \label{eq:47}
  \nabla^n_{X_1, \ldots, X_n} s = \nabla_{X_1} \nabla^{n-1}_{X_2,
    \ldots, X_n} s - \sum \mathop{}_{\!j} \nabla^{n-1}_{X_1, \ldots,
    \nabla_{X_1} X_j, \ldots, X_n} s,
\end{align}
with the obvious induction start given by the covariant derivative. It
is easily verified that this expression is tensorial in the vector
fields, so we get a map
\begin{align*}
  \nabla^n \colon C^\infty(M, TM^n)\to \mathcal{D}(M, E).
\end{align*}
For any tensor field $T_n \in C^\infty(M, TM^n)$, the symbol of
$\nabla^n_{T_n}$ is given by the symmetrization $\Sym(T_n) \in
C^\infty(M, S^n(TM))$ of $T_n$.  If $D \in \mathcal{D}(M, E)$ is an
operator, of order at most $n$, with principal symbol $\sigma_P(D) =
S_n \in C^\infty(M, S^n(TM))$, then the operator $D - \nabla^n_{S_n}$
is of order at most $n-1$, since its principal symbol
vanishes. Inductively, it follows that the operator $D$ can be written
uniquely in the form
\begin{align}
  \label{eq:53}
  D = \nabla^n_{S_n} + \nabla^{n-1}_{S_{n-1}} + \cdots + \nabla_{S_1}
  + S_0,
\end{align} where $S_d \in C^\infty(M, S^d(TM))$ is called the
\emph{symbol of order $d$} and gives rise to a map
\begin{align*}
  \sigma_d \colon \mathcal{D}(M, E) \to C^\infty(M, S^d(TM)).
\end{align*}
Any finite order differential operator on $E$ is uniquely determined
by the values of these symbol maps. In fact, through the expression
\eqref{eq:53}, a choice of symbols specifies a differential operator
on any vector bundle with connection, and in particular on functions.

We will also need the notion of divergence of vector fields and more
general contravariant tensors. Recall that the divergence $\delta X$
of a vector field $X$ on $M$ is defined in terms of the Lie derivative
and volume form by the equation $\mathcal{L}_X \omega^m = (\delta X)
\omega^m$. Although the divergence of a vector field only depends on
the symplectic volume, and not on the K\"ahler metric itself, a simple
computation reveals that the divergence can be calculated using the
Levi-Civita connection by the formula
\begin{align}
  \label{eq:32}
  \delta X = \Tr \nabla X = \nabla_a X^a,
\end{align}
in which the independence of the K\"ahler structure is perhaps not so
evident. The Laplace-{de Rham} operator on functions can be expressed
in terms of the divergence by
\begin{align}
  \label{eq:33}
  \smash{\Delta f = -2i \delta X'_f},
\end{align}
where $X'_f = \bar \d f \trdot \tilde \omega$ denotes the (1,0)-part
of the Hamiltonian vector field associated with the function $f \in
C^\infty(M)$.

The formula \eqref{eq:32} generalizes to tensors of higher degree. For
vector fields $X_1, \ldots, X_n$ om $M$, we define
\begin{align*}
  \delta(X_1 \sotimes \cdots \sotimes X_n) = \delta (X_1) X_2 \sotimes
  \cdots \sotimes X_n + \sum \mathop{}_{\!j} X_2 \sotimes
  \cdots \sotimes \nabla_{X_1} X_j \sotimes \cdots \sotimes X_n.
\end{align*}
This defines a map $\delta \colon C^\infty(M, TM^{n})
\to C^\infty(M, TM^{n-1})$, also called the divergence, which does
depend on the K\"ahler structure.

\pagebreak[1]
The generalization of divergence to sections of the endomorphism
bundle of the tangent bundle will also be convenient. If $\alpha \in
\Omega^1(M)$ is a one-form and $X$ is a vector field, we define
\begin{align*}
  \delta (X \otimes \alpha) = \delta (X) \alpha + \nabla_X \alpha,
\end{align*}
which gives a map $\delta \colon C^\infty(M, \End(TM)) \to
\Omega^1(M)$.

Finally, for any bivector field $B \in C^\infty(M, TM^2)$, we
introduce the second-order differential operator,
\begin{align*}
  \Delta_B = \nabla^2_B + \nabla_{\delta B},
\end{align*}
which will appear repeatedly throughout the paper.

\subsection*{Infinitesimal Deformations}

For a smooth family of K\"ahler structures, we can take its derivative
along a vector field $V$ on $\mathcal{T}$ to obtain a map
\begin{align*}
  V[J] \colon \mathcal{T} \to C^\infty(M, \End(TM)).
\end{align*}
Differentiating the identity $J^2 = - \Id$, we see that $V[J]$ and $J$
anti-commute,
\begin{align}
  \label{eq:19}
  V[J]J + J V[J] = 0,
\end{align}
so $V[J]_\sigma$ interchanges types on the K\"ahler
manifold $M_\sigma$. Therefore, it splits as
\begin{align}
  \label{eq:20}
  V[J] = V[J]' + V[J]'',
\end{align}
where $V[J]'_\sigma \in C^\infty(M, T'\!M_\sigma \otimes
T''\!M^*_\sigma)$ and $V[J]''_\sigma \in C^\infty(M, T''\!M_\sigma
\otimes T'\!M^*_\sigma)$ is its conjugate. Notice that this splitting
occurs for any infinitesimal deformation of an almost complex
structure on $M$ and defines an almost complex structure on the space
of almost complex structures on $M$.

Differentiating the integrability condition on $J$, expressed through
the vanishing of the Nijenhuis tensor, reveals that $V[J]' \in
\Omega^{0,1}(M, T'\!M)$ satisfies the holomorphicity condition $\bar
\d V[J]' = 0$, and the associated cohomology class in $H^1(M, T'\!M)$
is the Kodaira-Spencer class of the deformation (see \cite{MR815922}).

Define a bivector field $\tilde G(V) \in C^\infty(M, TM_\setC \otimes
TM_\setC)$ by the relation
\begin{align*}
  V[J] = \tilde G(V) \trdot \omega,
\end{align*}
for any vector field $V$ on $\mathcal{T}$. Differentiating the
identity $\tilde g = - J \trdot \tilde \omega $ along $V$, we get
\begin{align}
  \label{eq:15}
  V[\tilde g] = - V[J] \trdot \tilde \omega = - \tilde G(V),
\end{align}
and since $\tilde g$ is symmetric, this implies that $\tilde G(V)$ is
a symmetric bivector field. Furthermore, the combined types of $V[J]$
and $\omega$ yield a decomposition,
\begin{align*}
  \tilde G(V) = G(V) + \bar G(V),
\end{align*}
where $G(V)_\sigma \in C^\infty(M, S^2(T'\!M_\sigma))$ and $\bar
G(V)_\sigma \in C^\infty(M, S^2(T''\!M_\sigma))$. In other words, the
real symmetric bivector field $\tilde G(V)$ has no (1,1)-part.  The
variation of the K\"ahler metric is obtained by differentiating the
identity $g = \omega \trdot J$, which yields
\begin{align*}
  V[g] = \omega \trdot V[J] = \omega \trdot \tilde G(V) \trdot \omega
  = g \trdot \tilde G(V) \trdot g.
\end{align*}
We shall also need the variation of the Levi-Civita connection, which
is the tensor field
\begin{align*}
  V[\lcc] \in C^\infty(M, S^2(TM^*) \otimes TM)
\end{align*}
given by (see \cite{MR867684} Theorem 1.174)
\begin{align*}
  2g( V[\lcc]_{X} Y, Z) = \nabla_{X} (V[g])(Y, Z) + \nabla_{Y}
  (V[g])(X, Z) - \nabla_{Z} (V[g]) (X, Y),
\end{align*}
for any vector fields $X, Y$ and $Z$ on $M$. In index notation, this
translates to
\begin{align}
  \label{eq:8}
  2V[\lcc]_{ab}^c = \nabla_a \tilde G(V)^{cu} g_{ub} + g_{au} \nabla_b
  \tilde G(V)^{uc} - g_{au} \tilde g^{cw} \nabla_w \tilde G(V)^{uv}
  g_{vb},
\end{align}
and we remark that the trace $V[\lcc]_{xb}^x$ of this tensor
vanishes. Indeed, we get that
\begin{align}
  \label{eq:48}
  V[\lcc]_{xb}^x = \nabla_x \tilde G(V)^{xu} g_{ub} + g_{xu} \nabla_b
  \tilde G(V)^{ux} - g_{xu} \tilde g^{xw} \nabla_w \tilde G(V)^{uv}
  g_{vb} = 0,
\end{align}
where the first and last term cancel, and the middle term vanishes
because $\tilde G(V)$ has no part of type (1,1), which is the type of
the metric. 

We will also need to know the variation of the Ricci curvature, which
will result from the Bianchi identity of a certain line bundle
associated with the family of complex structures.

\subsection*{The Canonical Line Bundle of a Family}

For a family $J$ of K\"ahler structures, we can consider the vector bundle,
\begin{align*}
  \hat T'\!M \to \mathcal{T} \times M,
\end{align*}
with fibers $\hat T'\!M_{(\sigma, p)} = T'_pM_\sigma$ given by the
holomorphic tangent spaces of $M$. Throughout the paper, we shall
generally use a hat in the notation to indicate that we are working
over the product $\mathcal{T} \times M$. Following this convention,
the exterior differential on $\mathcal{T} \times M$ is denoted by
$\hat d$, whereas the differential on $\mathcal{T}$ is denoted by
$d_{\scriptscriptstyle \mathcal{T}}$ and by $d$ on $M$.

The K\"ahler metric induces a Hermitian structure $\hat h^{T'\!M}$ on
$\hat T'\!M$, and the Levi-Civita connection gives a compatible
partial connection along the directions of $M$. We can extend this
partial connection to a full connection $\hat \nabla^{T'\!M}$ on $\hat
T'\!M$ in the following way. If $Z \in C^\infty(\mathcal{T} \times M,
\hat T'\!M)$ is a smooth family of sections of the holomorphic tangent
bundle, and $V$ is a vector field on $\mathcal{T}$, then we define
\begin{align*}
  \hat \nabla_V Z = \pi^{1,0} V[Z].
\end{align*}
In other words, we regard $Z$ as a smooth family of sections of the
complexified tangent bundle $TM_\setC$, and then we simply
differentiate $Z$ along $V$ in this bundle, which does not depend on
the point in $\mathcal{T}$, and project the result back onto the
holomorphic tangent bundle.

Clearly, the connection $\hat \nabla^{T'\!M}$ preserves the
Hermitian structure in the directions of $M$, since it is induced by
the Levi-Civita connection. Moreover, if $V$ is a vector field on
$\mathcal{T}$, and $X$ and $Y$ are sections of $\hat T'\!M$, we get
that
\begin{align*}
  V[\hat h^{T'\!M}(X,Y)] = V[g(X,\bar Y)] &= V[g](X, \bar Y) + g(V[X],
  \bar Y) + g(X, \overline{V[Y]}) \\ &= h(\hat \nabla_V X, Y) + h(X,
  \hat \nabla_V Y),
\end{align*}
since the $(1,1)$-part of $V[g]$ vanishes. It follows that $\hat
\nabla^{T'\!M}$ preserves the Hermitian structure on $\hat T'\!M$.

Now consider the line bundle
\begin{align*}
  \hat K = \bwedge^m \hat T'\!M^* \to \mathcal{T} \times M,
\end{align*}
which will be referred to as the canonical line bundle of the family
of K\"ahler structures. As usual, the Hermitian structure and
connection on $\hat T'\!M$ induce a Hermitian structure $\hat h^K$ and
a compatible connection $\hat \nabla^K$ on $\hat K$.
%\subsection*{Curvature of the Canonical Line Bundle}
The curvature of $\hat \nabla^K$ was calculated in \cite{MR2928088}
and will be recalled below, but before stating it, we introduce the
following important notation. For any vector fields $V$ and $W$ on
$\mathcal{T}$, we define $\Theta \in \Omega^2(\mathcal{T}, S^2(TM))$
by
\begin{align*}
  \Theta(V,W) = \Sym(\tilde G(V) \trdot \omega \trdot \tilde G(W)),
\end{align*}
where $\Sym$ denotes symmetrization. We also give a name to the
metric trace of the symmetric bivector field $\Theta(V,W)$ and define
\begin{align}
  \label{eq:22}
  \begin{aligned}
    \theta(V, W) &= -\frac{1}{4} g (\Theta(V,W)) = -\frac{1}{4} g_{uv}
    \Theta(V,W)^{uv}.
  \end{aligned}
\end{align}
Clearly, this defines a real two-form $\theta \in
\Omega^2(\mathcal{T}, C^\infty(M))$ on $\mathcal{T}$ with values in
smooth functions on $M$. 

We note that the two-form $\Theta$ over $\mathcal{T}$ is
exact. To see this, we take the variation of $G(V) =
\pi^{2,0} ( \tilde G(V) ) = (\pi^{1,0} \sotimes \pi^{1,0}) \tilde
G(V)$ along $W$ to get
\begin{align}
  \label{eq:5}
  \begin{aligned}
    2W[G(V)] &= i G(V) \trdot \omega \trdot \bar G(W) - i \bar G(W)
    \trdot \omega \trdot G(V) - \pi^{2,0}( WV[\tilde g]) \\ &= 2i
    \Sym(G(V) \trdot \omega \trdot \bar G(W)) - \pi^{2,0}( WV[\tilde
    g]),
  \end{aligned}
\end{align}
which in turn shows that
\begin{align}
  \label{eq:3}
  V[G(W)] - W[G(V)] 
  &= -i \Sym(\bar G(V) \trdot \omega \trdot G(W)) - i\Sym(G(V)
  \trdot \omega \trdot \bar G(W)) = -i \Theta (V,W),
\end{align}
for commuting
vector fields $V$ and $W$ on $\mathcal{T}$. This can be rephrased as
\begin{align}
  \label{eq:44}
  d_{\scriptscriptstyle \mathcal{T}} G = - i\Theta, 
\end{align}
where $G$ is viewed as a one-form in $\Omega^1(\mathcal{T}, S^2(T'\!M))$.

The following proposition gives the curvature of the canonical line
bundle of a family of K\"ahler structures and is proved in
\cite{MR2928088}.
\begin{proposition}
  \label{prop:6}
  The curvature of $\hat \nabla^K$ is given by
  \begin{align*}
    &F_{\hat \nabla^K}(X,Y) = i\rho(X, Y), \qquad F_{\hat \nabla^K}(V,
    X) = \frac{i}{2} \delta \tilde G(V) \trdot \omega \trdot X, \qquad
    F_{\hat \nabla^K}(V, W) = i \theta(V, W),
 \end{align*}
  for any vector fields $X, Y$ on $M$ and $V, W$ on $\mathcal{T}$.
\end{proposition}

By applying the Bianchi identity to the connection $\hat \nabla^K$,
and using the formulas of \Fref{prop:6}, we get three useful results.
The first is the fact that the two-form $\theta \in
\Omega^2(\mathcal{T}, C^\infty(M))$ is closed, which is a trivial
reformulation of the Bianchi identity for three vector fields on
$\mathcal{T}$. By applying the
Bianchi identity to two vector fields $V$ and $W$ on $\mathcal{T}$,
and one vector field on $M$, we get
\begin{align}
  \label{eq:9876}
  d\theta(V, W) 
  = \smash{\frac{1}{2} W[\delta \tilde G(V)] \trdot
    \omega - \frac{1}{2} V[\delta \tilde G(W)] \trdot \omega}.
\end{align}
Finally, the Bianchi identity for two vector fields on $M$ and one on
$\mathcal{T}$ gives following important formula
for the variation of the Ricci form,
\begin{align}
  \label{eq:1234}
  V[\rho] = \smash{\frac{1}{2} d (\delta\tilde G(V) \trdot \omega),}
\end{align}
for any vector field $V$ on $\mathcal{T}$.  As an immediate
consequence of \eqref{eq:1234}, we get the following simple formula
for the variation of the scalar curvature
\begin{align}
  \label{eq:10}
  \begin{aligned}
    V[s] &= V[\rho]_{uv} \tilde \omega^{vu} = \frac{1}{2} \nabla_u
    (\delta \tilde G(V))^w \omega_{wv} \tilde \omega^{vu} -
    \frac{1}{2} \nabla_v (\delta \tilde G(V))^w \omega_{wu} \tilde
    \omega^{vu}= \delta \delta \tilde G(V),
  \end{aligned}
\end{align}
for any vector field $V$ on $\mathcal{T}$.

\subsection*{Holomorphic Families of K\"ahler Structures}

In case the manifold $\mathcal{T}$ is itself a complex manifold, we
can require the family $J$ to be a holomorphic map from $\mathcal{T}$
to the space of complex structures. This is made precise by the
following definition, which uses the splitting \eqref{eq:20} of
$V[J]$.

\begin{definition}
  \label{def:3}
  Suppose that $\mathcal{T}$ is a complex manifold, and that $J$ is a
  family of complex structures on $M$, parametrized by
  $\mathcal{T}$. Then $J$ is \emph{holomorphic} if
  \begin{align*}
    V'[J] = V[J]' \qquad \text{and} \qquad V''[J] = V[J]'',
  \end{align*}
  for any vector field $V$ on $\mathcal{T}$.
\end{definition}

If $I$ denotes the integrable almost
complex structure on $\mathcal{T}$ induced by its complex structure,
then we get an almost complex structure $\hat J$ on $\mathcal{T}
\times M$ defined by
\begin{align*}
  \hat J(V \oplus X) = IV \oplus J_\sigma X, \qquad V \oplus X \in
  T_{(\sigma,p)} (\mathcal{T} \times M).
\end{align*}
The following proposition gives another characterization of
holomorphic families \cite{MR2928088}.
\begin{proposition}
  \label{prop:9}
  The family $J$ is holomorphic if and only if $\hat J$ is integrable.
\end{proposition}

By this proposition, a holomorphic family induces a complex structure
on the product manifold $\mathcal{T} \times M$. Clearly, the
projection $\pi_{\scriptscriptstyle \mathcal{T}} \colon \mathcal{T}
\times M \to \mathcal{T}$ is a holomorphic map, and its differential
is the projection $d\pi_{\scriptscriptstyle \mathcal{T}} \colon \hat
T' \mathcal{T} \oplus \hat T'\!M \to T'\mathcal{T}$, where $\hat
T'\mathcal{T}$ is the pullback of $T'\mathcal{T}$ by
$\pi_{\scriptscriptstyle \mathcal{T}}$. Since the bundle $\hat T'\!M$
over $\mathcal{T} \times M$ is the kernel of this map, it has the
structure of a holomorphic vector bundle, and it is easily verified
that the connection $\hat \nabla^{T'\!M}$ is compatible with this
holomorphic structure. Since the connection also preserves the
Hermitian structure, it must be the Chern connection.

Holomorphicity has several useful consequences. First of all, it implies
that
\begin{align*}
  \tilde G(V') = V'[J] \trdot \tilde \omega = V[J]' \trdot \tilde
  \omega = G(V),
\end{align*}
and similarly $\tilde G(V'') = \bar G(V)$. This means that $G$,
viewed as a one-form over the complex manifold $\mathcal{T}$, has type
(1,0) and that $\Theta$ and $\theta$ have type (1,1). These important
facts will be used without reference going forward.

% By differentiating the identities $V'[J]\pi^{1,0} = 0$ and $W''[J]
% \pi^{0,1} = 0$ along $W''$ and $V'$, respectively, we get that
% \begin{align*}
%   W'' V'[J] \pi^{1,0} = \frac{i}{2} V'[J] W''[J] \qquad \text{and}
%   \qquad V'W''[J] \pi^{0,1} = - \frac{i}{2} W''[J] V'[J].
% \end{align*}
% Adding these identies, we get that
% \begin{align}
%   \label{eq:42}
%   W''V'[J] = \smash{\frac{i}{2}} \big [V'[J], W''[J] \big ],
% \end{align}
Finally, for commuting vector fields $V'$ and $W''$, the identity
\eqref{eq:3} reduces to
\begin{align}
  \label{eq:43}
  W''[G(V)] = \frac{i}{2} G(V) \trdot \omega \trdot \bar G(W) -
  \frac{i}{2} \bar G(W) \trdot \omega \trdot G(V)
  = i \Theta(V',W'').
\end{align}
This expression for the second-order variation of the complex
structure will prove very useful in later calculations, but we
emphasize the fact that it only holds for commuting vector fields.

\subsection{Rigid Families of K\"ahler Structures}

The following rather serious assumption on a family of K\"ahler
structures turns out to be crucial to the construction of the Hitchin
connection as well as the calculation of its curvature.
\begin{definition}
  \label{def:4}
  A family of K\"ahler structures is called \emph{rigid} if
  \begin{align}
    \label{eq:28}
    \nabla_{X''} G(V) = 0,
  \end{align}
  for all vector fields $V$ on $\mathcal{T}$ and $X$ on $M$.
\end{definition}

In other words, the family $J$ is rigid if $G(V)$ is a holomorphic
section of $S^2(T'\!M)$, for any vector field $V$ on
$\mathcal{T}$. For examples of rigid families in a basic setting, we
refer to \cite{MR2928088}.  By differentiating the rigidity condition
\eqref{eq:28} along $\mathcal{T}$, we get the following crucial
result.
\begin{proposition}
  \label{prop:14}
  Any rigid family of K\"ahler structures satisfies the symmetry property
  \begin{align}
    \label{eq:83}
    \Sym (G(V) \trdot \nabla G(W)) = \Sym (G(W) \trdot \nabla G(V)),
  \end{align}
  for any vector fields $V$ and $W$ on $\mathcal{T}$.
\end{proposition}

\begin{proof}
  Throughout the proof, let $V$ and $W$ be commuting vector fields on
  $\mathcal{T}$. By differentiaing the holomorphicity condition on
  the bivector field $G(V)$ along $W$, we obtain
  \begin{align}
    \label{eq:1}
    \begin{aligned}
      0 &= W[\nabla_{a''} G(V)^{bc}] \\ &=
      W[\pi^{0,1}]^u_a\nabla_{u} G(V)^{bc} + \nabla_{a''}
      W[G(V)]^{bc} \\ & \qquad + W[\nabla]^b_{a''u} G(V)^{uc} +
      W[\nabla]^c_{a''u} G(V)^{bu}.
    \end{aligned}
  \end{align}
  Let us work out each of these terms individually. Using rigidity, the
  first term reduces to
  \begin{align*}
    2W[\pi^{0,1}]^u_a\nabla_{u} G(V)^{bc} = i \tilde G(W)^{uv}
    \omega_{va} \nabla_u G(V)^{bc} = - g_{av} G(W)^{uv} \nabla_u
    G(V)^{bc}.
  \end{align*}
  For the second term of \eqref{eq:1}, we simply apply \eqref{eq:5}
  and rigidity to obtain
  \begin{align*}
    2\nabla_{a''} &(W[G(V)]^{bc}) \\ &= - \nabla_a (\bar G(W)^{bu})
    g_{uv} G(V)^{vc} - G(V)^{bu} g_{uv} \nabla_a (\bar G(W)^{vc}) -
    \nabla_{a''}( (\pi^{2,0})^{bc}_{uv} WV[\tilde g^{uv}]).
  \end{align*}
  Finally, for the last two terms of \eqref{eq:1}, we can apply
  \eqref{eq:8} in combination with rigidity of the family to obtain
  \begin{align*}
    2W[\nabla]^b_{a''u} G(V)^{uc} = g_{av}\nabla_u (G(W)^{vb})
    G(V)^{uc} + \nabla_a (\bar G(W)^{bv}) g_{vu} G(V)^{uc}.
  \end{align*}
  Substituting these expressions into \eqref{eq:1}, and cancelling
  terms, we get that
  \begin{align*}
    0 &= g_{av}\nabla_u (G(W)^{vb}) G(V)^{uc} + g_{av}\nabla_u
    (G(W)^{vc}) G(V)^{ub} - g_{av} G(W)^{uv} \nabla_u G(V)^{bc} \\ &
    \quad - \nabla_{a''}( (\pi^{2,0})^{bc}_{uv} WV[\tilde g^{uv}]).
  \end{align*}
  Raising an index, this implies that
  \begin{align*}
    3\Sym (&G(V) \trdot \nabla G(W))^{abc} \\ &= G(V)^{au} \nabla_u
    G(W)^{bc} + G(W)^{au} \nabla_u G(V)^{bc} +\tilde
    g^{aw}\nabla_{w''}( (\pi^{2,0})^{bc}_{uv} WV[\tilde g^{uv}]),
  \end{align*}
  and clearly this is symmetric in $V$ and $W$ because these were
  chosen to commute. This proves the proposition.
\end{proof}

Repeatedly taking the divergence of the identity \eqref{eq:83} yields
similar identities for contravariant tensors of lower degree. These
will play a crucial role in the calculation of the curvature of the
Hitchin connection, so we introduce special notation for them.
Inspired by \Fref{prop:14}, we define a two-form $\Gamma_3$ on
$\mathcal{T}$, with values in sections of the third symmetric power of
the holomorphic tangent bundle over $M$, by
\begin{align*}
  \Gamma_3(V,W) = \Sym (G(V) \trdot \nabla G(W)),
\end{align*}
for any vector fields $V$ and $W$ on $\mathcal{T}$. Then
\Fref{prop:14} ensures that $\Gamma_3$ actually defines a symmetric
two-form on $\mathcal{T}$. 

Clearly, the symmetry of the two-form over
$\mathcal{T}$ is not affected by taking the divergence of the
tri-vector field part of $\Gamma_3$ over $M$. In other words, $\delta
\Gamma_3$ defines a symmetric two-form on $\mathcal{T}$ with values in
the second symmetric power of the holomorphic tangent bundle on
$M$. It is given by
\begin{align*}
  3 \delta \Gamma_3(V,W) &= \Delta_{G(V)} G(W)+ 2 \Sym( G(V) \trdot
  \nabla \delta G(W)) + 2 \Sym(\nabla_u G(V)^{av} \nabla_v G(W)^{ub}),
\end{align*}
for any vector fields $V$ and $W$ on $\mathcal{T}$. In contrast to
the first two terms of this expression, the last term is obviously
symmetric in $V$ and $W$. This leads us to define $\Gamma_2 \in
\Omega^2(\mathcal{T}, C^\infty(M, S^2(T'M))$ by
\begin{align*}
  \Gamma_2(V,W) = \Delta_{G(V)} G(W) + 2
  \Sym( G(V) \trdot \nabla \delta G(W)),
\end{align*}
which is also a symmetric two-form on $\mathcal{T}$ and encodes the
interesting part of the symmetry statement for $\delta \Gamma_3$.

\pagebreak[1] Repeatedly taking the divergence, and removing obviously symmetric
parts, we get the following important proposition
\begin{proposition}
  \label{prop:2}
  The four two-forms $\Gamma_j$ on $\mathcal{T}$, with values in
  symmetric contravariant tensors on $M$, defined by
  \begin{align*}
    \Gamma_3(V,W) &= \Sym (G(V) \trdot \nabla G(W)) \\
    \Gamma_2(V,W) &= 2 \Sym( G(V) \trdot \nabla \delta G(W)) +
    \Delta_{G(V)} G(W) \\
    \Gamma_1(V,W) &= 2\Delta_{G(V)} \delta G(W) + G(V) \trdot
    d\delta\delta G(W) + \nabla_w(G(V)^{uv}) \nabla^2_{uv}(G(W)^{wa}) \\
    \Gamma_0(V,W) &= \Delta_{G(V)} \delta \delta G(W) + \nabla_w
    (G(V)^{uv}) \nabla^2_{uv} \delta G(W)^{w},
  \end{align*}
  are all symmetric in the vector fields $V$ and $W$ on $\mathcal{T}$.
\end{proposition}

\subsection*{Families of Ricci Potentials}

Suppose that the first Chern class of $(M, \omega)$
is represented by the symplectic form, that is
\begin{align}
  \label{eq:54}
  c_1(M, \omega) = \ccint [ \tfrac{\omega}{2\pi} ],
\end{align}
for some integer $\ccint \in \setZ$.  Now the first Chern class is
also represented by the Ricci form $\frac{\rho}{2\pi}$, so the
difference between the forms $\rho$ and $\ccint \omega$ is exact. A
smooth real function $F \in C^\infty(\mathcal{T} \times M)$, which can
be viewed as a smooth map $F \colon \mathcal{T} \to C^\infty(M)$, is
called a \emph{family of Ricci potentials} if it satisfies
\begin{align}
  \label{eq:23}
  \rho_\sigma = \ccint \omega + 2i \d_\sigma \bar \d_\sigma F_\sigma,
\end{align}
for every point $\sigma \in \mathcal{T}$.

Any two Ricci potentials differ by a global pluriharmonic real
function. If we assume that $H^1(M, \setR)$ vanishes, such a function
is globally the real part of a holomorphic function, so a family of
Ricci potential is then uniquely determined up to a function on
$\mathcal{T}$ if the family of
K\"ahler structures does not allow any non-constant holomorphic
functions on $M$.

The existence of a Ricci potential is clearly a global issue over $M$,
as the local $\d \bar \d$-lemma ensures local existence around any
point on $M$ when \eqref{eq:54} holds. If the manifold $M$ is compact,
the global $\d \bar \d$-lemma from Hodge theory ensures the existence
of a Ricci potential, using the fact that the K\"ahler form $\omega$
is harmonic. In this case, the Ricci potential is certainly unique up to a
function on $\mathcal{T}$, and by imposing zero average over $M$, 
\begin{align}
  \label{eq:24}
  \int_M F \omega^m = 0.
\end{align}
we can fix it uniquely.

The following proposition gives an important identity, involving the
variation of a family of Ricci potentials.

\begin{proposition}
  \label{prop:10}
  Suppose that $M$ is a symplectic manifold with $H^1(M, \setR) = 0$
  and $c_1(M, \omega) = \ccint [\frac{\omega}{2\pi}]$, and let $J$ be a
  holomorphic family of K\"ahler structures on $M$, none of which
  admit non-constant holomorphic functions on $M$. Then
  \begin{align}
    \label{eq:41}
    4i \bar \d V'[F] = \delta G(V) \trdot \omega + 2 dF \trdot G(V)
    \trdot \omega,
  \end{align}
  for any family of Ricci potentials $F$ and any vector field $V$ on
  $\mathcal{T}$.
\end{proposition}

\begin{proof}
  By differentiating the identity \eqref{eq:23} in the direction of
  $V'$, we get
  \begin{align*}
    V'[\rho] & = - d (dF \trdot G(V) \trdot \omega) + 2i \d \bar \d V'[F],
  \end{align*}
  and by using \eqref{eq:1234} on the left-hand side, this yields
  \begin{align*}
    d( \delta G(V) \trdot \omega) + 2d(dF \trdot G(V) \trdot \omega) -
    4i d \bar \d V'[F] = 0.
  \end{align*}
  On one hand, it follows that the one-form
  \begin{align*}
    \delta G(V) \trdot \omega + 2 dF \trdot G(V) \trdot \omega - 4i
    \bar \d V'[F]
  \end{align*}
  is closed, and hence exact by the assumption $H^1(M, \setR) = 0$.  On
  the other hand, it is of type (0,1), so it cannot be exact unless it
  is zero, because we assumed that none of the K\"ahler structures
  admit non-constant holomorphic functions. This proves the lemma.
\end{proof}
% On a few occasions, we will use \eqref{eq:41} on the equivalent form
% \fxnote{are we actually using this}
% \begin{align}
%   \label{eq:11}
%   X'_{V'[F]} = - \frac{i}{4} \delta G(V) - \frac{i}{2} G(V) \trdot dF,
% \end{align}
% so we state it here for easy reference. 
Another equivalent form of \eqref{eq:41} is the following,
\begin{align*}
  4i V'[\bar \d F] = \delta G(V) \trdot \omega,
\end{align*}
and if we combine this with \eqref{eq:9876}, for commuting vector
fields $V'$ and $W''$ on $\mathcal{T}$, we get
\begin{align*}
  2d \theta (V', W'') 
  &= W''[\delta G(V) \trdot \omega] - V'[\delta \bar G(W) \trdot
  \omega]
  \\ &= 4i W''V'[\bar \d F] - 4i V'W''[\d F] \\ &= 4i V'W''[d F] \\ &=
  4i \d_{\scriptscriptstyle \mathcal{T}} \bar \d_{\scriptscriptstyle
    \mathcal{T}} F (V',W'').
\end{align*}
Since the family is holomorphic, the form $\theta$ has type (1,1) on
$\mathcal{T}$, so we have shown
\begin{proposition}
  \label{prop:1}
  In the setting of \Fref{prop:10}, any family of Ricci potentials
  satisfies
  \begin{align*}
    \theta - 2i \d_{\scriptscriptstyle \mathcal{T}} \bar \d_{\scriptscriptstyle
    \mathcal{T}} F \in \Omega^{1,1}(\mathcal{T}).
  \end{align*}
  In other words, the form takes values in constant functions on $M$.
\end{proposition}

This ends the general discussion of families of K\"ahler
structures. In the next section, we discuss general aspects of
geometric quantization, and in particular the need for a choice of
auxiliary polarization in the construction. Understanding the
effects of this choice naturally leads us to
consider families of K\"ahler structures, and ultimately to the Hitchin
connection relating the different choices. The results from this
section will play a fundamental role in the discussion.

\chapter{Geometric Quantization}
\label{cha:geom-quant}

In very broad terms, geometric quantization concerns the passage from
classical mechanics to quantum mechanics. It aims to produce a Hilbert
space of quantum states from a classical phase space, in the form of a
symplectic manifold, and a self-adjoined operator from a classical
observable, in the form of a function on the classical phase space.

In the following, we shall briefly review the basic notions from
geometric quantization relevant to us. A key role is played by an
auxiliary choice of polarization, which is often chosen to be
K\"ahler. For a broader treatment of geometric quantization, the
reader is referred to \cite{MR1183739} and
\cite{MR2151954}. 

After reviewing the elements of geometric quantization, we calculate
the commutators of certain differential operators acting on the
prequantum spaces. These will be relevant in later discussions of the
Hitchin connection and its curvature.

\subsection{Prequantization}

In geometric quantization, the Hilbert space of quantum states arises
as sections of a certain Hermitian line bundle over the classical
phase space. As a model for this classical phase space, we consider a
symplectic manifold $(M, \omega)$ of dimension $2m$. A prequantum line
bundle over the $M$ is a complex line bundle $\mathcal{L}$ endowed
with a Hermitian metric $h$ and a compatible connection $\nabla$ of
curvature
\begin{align*}
  F_{\nabla} = -i \omega.
\end{align*}

A symplectic manifold admitting a prequantum line bundle is called
prequantizable. Evidently, this is not the case for every symplectic
manifold. Indeed, the real first Chern class of a prequantum line
bundle is given by $c_1(\mathcal{L}) = \big
[\frac{\omega}{2\pi}\big ]$, leading us to the following necessary
condition for prequantizability, called the prequantum condition,
\begin{align}
  \label{eq:293485729}
  \big[ \tfrac{\omega}{2\pi} \big] \in \im \big (H^2(M, \setZ) \to
  H^2(M, \setR) \big).
\end{align}
This is, in fact, also sufficient to ensure the existence of a
prequantum line bundle, and the inequivalent prequantum
line bundles over $M$ are parametrized by $H^1(M, \mathrm{U}(1))$.

For any natural number $k$, called the \emph{level}, we consider the
\emph{prequantum space},
\begin{align*}
  \pqs = C^\infty(M, \mathcal{L}^k),
\end{align*}
of smooth sections of the $k$'th tensor power of the line bundle
$\mathcal{L}$. These sections play the role of wave functions in the
quantum theory. If $f \in C^\infty(M)$ is a function on $M$, the
corresponding \emph{prequantum operator}, acting on $\pqs$,
is defined by
\begin{align}
  \label{eq:9}
  P_k(f) = \frac{i}{k}\nabla_{X_f} + f,
\end{align}
where $X_f$ is the Hamiltonian vector field of the function $f$. The virtue
of \eqref{eq:9} is that the prequantum operators satisfy the
correspondence principle
\begin{align}
  \label{eq:16}
  \com{P_k(f)}{P_k(g)} = \frac{i}{k} P_k(\{f,g\}),
\end{align}
which is one of the distinctive features of a viable quantization.

From a physical perspective, the wave functions in $\pqs$
depend on twice the number of variables they should. A standard way to
remedy this is to pick an auxiliary polarization on $M$ and
consider the space of polarized sections of the line bundle. The
polarization can be by real or complex Lagrangian subspaces. In the
following, we will focus on the case of complex K\"ahler polarizations.

\subsection{K\"ahler Quantization}

From now on, we assume that the symplectic manifold admits a K\"ahler
structure, in the form of an integrable almost complex structure $J$
on $M$ which is compatible with the symplectic structure.
Since the K\"ahler form $\omega$ has type $(1,1)$, it follows that the
(0,1)-part of the connection on the prequantum line bundle
$\mathcal{L}$ defines a holomorphic structure. Therefore, we can define
the \emph{quantum space} to be the space of holomorphic (or polarized)
sections,
\begin{align*}
  \qs_J = H^0(M_J, \mathcal{L}^k) = \{s \in \pqs
  \mid \nabla_Z s = 0, \ \forall Z \in T''\!M_J\},
\end{align*}
which is a subspace of the prequantum space $\pqs$ of smooth
sections. If the manifold $M$ is compact, then $\qs_J$ is a
finite-dimensional space by standard theory of elliptic operators.

Unfortunalely, the prequantum operators do not in general preserve the
space of holomorphic sections.  A function $f \in C^\infty(M)$ is
polarized if the (1,0)-part of its corresponding Hamiltonian vector
field is holomorphic. In other words, the space of polarized functions
is given by
\begin{align*}
  C^\infty_J(M) = \{f \in C^\infty(M) \mid \nabla_Z X_f = 0, \ \forall
  Z \in T''\!M_J\}.
\end{align*}
Now, the operator $P_k(f)$ preserves $\qs_J$ if and only $f$ is a
polarized function. In fact, it is easily verified that a first-order
differential operator of the form $\frac{i}{k} \nabla_X + f$ on $\pqs$
preserves the subspace $\smash{\qs_J}$ if and only if $f$ is a
polarized function and $X' = X_f'$. This gives another justification
for the choice of prequantum operators. In fact, the (0,1)-part of the
vector field makes no difference to the action of $P_k(f)$ on $\qs_J$,
so if we define a polarized variation of the prequantum operators by
\begin{align}
  \label{eq:7}
  P_k\;\!\!\!'(f)_J = \frac{i}{k}\nabla_{X'_f} + f,
\end{align}
then these are essentially the only first-order operators with a
chance of preserving the subspace $\qs_J$ of
$\pqs$. For a real polarized function $f \in C_J^\infty(M)$,
the Hamiltonian vector field $X_f$ must be Killing for the K\"ahler
metric, effectively reducing the quantizable observables to an at most
finite-dimensional, and often trivial, space (see \cite{MR1183739}).

To get more quantizable observables, their quantization is modified in
the following way. The space $\smash{\qs_J}$ is in fact a closed
subspace of $\pqs$, and therefore we have the orthogonal projection
$\pi^{(k)}_J \colon \pqs \to \qs_J$. For $f \in C^\infty(M)$, we then
define the corresponding quantum operator by
\begin{align*}
  Q_{k}(f)_J = \pi^{(k)}_J \circ P_k(f).
\end{align*}
These operators do not form an algebra, but they satisfy a weaker form
of \eqref{eq:16} (at least if $M$ is compact) in the sense that
\begin{align}
  \label{eq:81}
  \Big \lVert [Q_k(f), Q_k(g)] - \frac{i}{k} Q_k(\{f,g\}) \Big \rVert
  = O(k^{-2}) \quad \text{as} \quad k \to \infty,
\end{align}
with respect to the operator norm on $\qs_J$.  The proof of
\eqref{eq:81} relies on the fact that these operators are Toeplitz
operators (see \cite{MR1301849}).

Although this quantization scheme gives a Hilbert space of the right
size, it still fails to produce the right answers on basic examples
from quantum mechanics. In the end, what really matters is the
spectrum of the operators, and if the above procedure is applied to
the one-dimensional harmonic oscillator, the quantization yields a
spectrum which differs from the correct one by a shift. To deal with
this problem, the so-called metaplectic correction can be
introduced. We shall not pursue this direction further in the
present paper, but refer the interested reader to \cite{MR2928088},
where the constrution of the Hitchin connection in the metaplectic
setting is discussed.

\subsection*{Commutators of Differential Operators}

The description of the Hitchin connection, and the calculation of its
curvature in particular, requires the calculation of commutators of a
number of differential operators on sections on the prequantum line
bundle $\mathcal{L}$ and its tensor powers.  Since the commutators of
even second-order operators are quite complicated, it will be
convenient to encode the operators through their symbols, as in
\eqref{eq:53}, and have general but explicit descriptions of the
symbols of the commutators in terms of the symbols of the
commutants. The followings lemmas give exactly such descriptions.

The very definition of the curvature of the line bundle
$\mathcal{L}^k$ implies the basic relation
\begin{align*}
  [\nabla_X, \nabla_Y]s = \nabla_{[X,Y]} s - ik \omega(X,Y) s
\end{align*}
for any vector fields $X,Y$ on $M$ and any smooth section $s \in
\pqs$. Things become a little more complicated when
second-order operators are introduced.
\begin{lemma}
  \label{lem:7}
  For any K\"ahler structure on $M$, any vector field $X$ on $M$, and
  any symmetric bivector field $B \in C^\infty(M, S^2(TM))$, we have
  the symbols
  \begin{align*}
    \sigma_2 \com{\nabla^2_B}{\nabla_X} &= 2 \Sym(B \trdot
    \nabla X) - \nabla_X B\\
    \sigma_1 \com{\nabla^2_B}{\nabla_X} &= \nabla^2_B X - 2ik B \trdot
    \omega \trdot X + B^{uv} R^a_{wuv} X^w\\
    \sigma_0 \com{\nabla^2_B}{\nabla_X} &= -ik\, \omega (B \trdot \nabla
    X) 
  \end{align*}
  for the commutator of the operators $\nabla_X$ and $\nabla^2_B$
  acting on $\pqs$.
\end{lemma}

\begin{proof}
  By straightforward calculation, we get
  \begin{align*}
    &B^{uv} \nabla^2_{uv} X^x \nabla_{x} s \\ & \qquad = B^{uv} X^x
    \nabla^3_{uvx} s + 2 B^{uv} \nabla_v(X^x) \nabla^2_{ux} s + B^{uv}
    \nabla^2_{uv}(X^x) \nabla_x s \\ & \qquad = B^{uv} X^x
    \nabla^3_{xuv} s - 2ik B^{uv} \omega_{vx} X^x \nabla_u s - B^{uv}
    X^x R^r_{uxv} \nabla_r s \\ & \qquad \qquad + B^{uv} \nabla_v
    (X^x) \nabla^2_{ux} s + B^{uv} \nabla_v (X^x) \nabla^2_{xu} s -
    ik B^{uv} \omega_{ux} \nabla_v (X^x) s \\ &\qquad \qquad + B^{uv}
    \nabla^2_{uv}(X^x) \nabla_x s,
  \end{align*}
  where we used the fact that the symplectic form $\omega$ is parallel
  with respect to the Levi-Civita connection.  To get an expression
  for the desired commutator, we subtract
  \begin{align*}
    X^x \nabla_x B^{uv} \nabla^2_{uv} s = X^x \nabla_x (B^{uv})
    \nabla^2_{uv} s + B^{uv} X^x \nabla^3_{xuv} s,
  \end{align*}
  and the stated symbols can easily be extracted from the result.
\end{proof}

Naturally, things get even more complicated for two second-order
operators.
\begin{lemma}
  \label{lem:8}
  For any K\"ahler structure on $M$ and any symmetric bivector fields
  $A, B \in C^\infty(M, S^2(TM))$, we have the symbols
  \begin{align*}
    \sigma_3 \com{\nabla^2_A}{\nabla^2_B} &= 2 \Sym (A \trdot \nabla
    B) - 2 \Sym(B \trdot \nabla A) \\[2pt]
    \sigma_2 \com{\nabla^2_A}{\nabla^2_B} &= \nabla^2_A B - \nabla^2_B A
    - 4ik \Sym(A\trdot \omega \trdot B) \\ & \qquad + 2 \Sym(A^{xy}
    R^a_{uxy} B^{ub}) - 2 \Sym(B^{uv} R^a_{xuv} A^{xb}) \\[2pt]
    \sigma_1 \com{\nabla^2_A}{\nabla^2_B} &= - 2i k A^{xy} \omega_{yu}
    \nabla_x (B^{ua}) + 2ik B^{uv} \omega_{vx} \nabla_u (A^{xa}) \\[2pt] &
    \qquad - A^{xy} \nabla_x (R^a_{yuv}) B^{uv} + B^{uv} \nabla_v
    (R^a_{uxy}) A^{xy} \\[2pt] & \qquad - \smash{\frac{4}{3}} A^{xy} R^a_{xuv}
    \nabla_y B^{uv} + \smash{\frac{4}{3}} B^{uv} R^a_{uxy} \nabla_v A^{xy} \\[2pt]
    \sigma_0 \com{\nabla^2_A}{\nabla^2_B} &= \frac{ik}{2} A^{xy}
    J_y^j R_{xuvj} B^{uv} - \frac{ik}{2} B^{uv} J_v^j R_{uxyj} A^{xy}
  \end{align*}
  for the commutator of the operators $\nabla^2_A$ and $\nabla^2_B$
  acting on  $\pqs$.
\end{lemma}

\begin{proof}
  Once again, the proof proceeds by straightforward calculation. We get
  \begin{align}
    \label{eq:13}
    \begin{aligned}
      &A^{xy}\nabla^2_{xy}B^{uv} \nabla^2_{uv} s = A^{xy} B^{uv}
      \nabla^4_{xyuv} s + 2 A^{xy} \nabla_y (B^{uv}) \nabla^3_{xuv} s
      + A^{xy} \nabla^2_{xy} (B^{uv}) \nabla^2_{uv} s.
    \end{aligned}
  \end{align}
  Focusing on the first term, we commute the indices $x$ and $y$ past
  $u$ and $v$.
  \begin{align*}
    & A^{xy} B^{uv} \nabla^4_{xyuv} s \\&\qquad = A^{xy} B^{uv} \bigl
    ( \nabla^4_{xuyv} - \nabla_x R^r_{yuv} \nabla_r - ik \omega_{yu}
    \nabla_{xv} \bigr ) s \\&\qquad = A^{xy} B^{uv} \bigl (
    \nabla^4_{xuvy} - ik \omega_{yv} \nabla^2_{xu} - \nabla_x R^r_{yuv}
    \nabla_r - ik \omega_{yu} \nabla^2_{xv} \bigr ) s \\&\qquad =
    A^{xy} B^{uv} \bigl ( \nabla^4_{uxvy} - R^r_{xuv} \nabla^2_{ry} -
    R^r_{xuy} \nabla^2_{vr} - ik \omega_{xu} \nabla^2_{vy}
    \\&\qquad\qquad - ik \omega_{yv} \nabla^2_{xu} - \nabla_x R^r_{yuv}
    \nabla_r - ik \omega_{yu} \nabla_{xv} \bigr ) s \\&\qquad = A^{xy}
    B^{uv} \bigl (\nabla^4_{uvxy} - \nabla_u R^r_{xvy} \nabla_r - ik
    \omega_{xv} \nabla^2_{uy} \\&\qquad\qquad - R^r_{xuv}
    \nabla^2_{ry} - R^r_{xuy} \nabla^2_{vr} - ik \omega_{xu}
    \nabla^2_{vy} - ik \omega_{yv} \nabla^2_{xu} - \nabla_x R^r_{yuv}
    \nabla_r - ik \omega_{yu} \nabla^2_{xv} \bigr ) s,
  \end{align*}
  where we used the fact that $\omega$ is parallel, but otherwise just
  added the curvature terms. Expanding by the Leibniz rule, using
  symmetries, and collecting terms, this can be rewritten as
  \begin{align}
    \label{eq:37}
    \begin{aligned}
      & A^{xy} B^{uv} \nabla^4_{xyuv} s \\&\quad = A^{xy} B^{uv}
      \bigl (\nabla^4_{uvxy} - R^r_{xuv} \nabla^2_{ry} - R^r_{xuv}
      \nabla^2_{yr} + R^r_{uxy} \nabla^2_{vr} + R^r_{uxy}
      \nabla^2_{rv} \\&\quad\quad - 2 ik \omega_{xu} \nabla^2_{yv} -
      2 ik \omega_{xu} \nabla^2_{vy} + \nabla_u (R^r_{vxy}) \nabla_r -
      \nabla_x (R^r_{yuv}) \nabla_r - ik R^r_{uxy} \omega_{vr} \bigr )
      s,
    \end{aligned}
  \end{align}
  where the order of differentiation was interchanged for the term
  $A^{xy} B^{uv} R^r_{xuy} \nabla^2_{vr}s$.  The second term of
  \eqref{eq:13} can be rewritten as
  \begin{align}
    \label{eq:35}
    \begin{aligned}
      &2 A^{xy} \nabla_y (B^{uv}) \nabla^3_{xuv} s \\&\qquad=
      \frac{2}{3} A^{xy} \nabla_y (B^{uv}) \bigr( \nabla^3_{xuv} + 2
      \nabla^3_{uxv}  - 2 R^r_{xuv} \nabla_r - i 2k\, \omega_{xu}
      \nabla_v \bigl) s \\&\qquad = \frac{2}{3} A^{xy} \nabla_y
      (B^{uv}) \big( \nabla^3_{xuv} + \nabla^3_{uxv} + \nabla^3_{uvx}
      - 2 R^r_{xuv} \nabla_r - i 3 k\, \omega_{xu} \nabla_v \big) s.
    \end{aligned}
  \end{align}
  Analagous to \eqref{eq:13}, the other term of the commutator
  $\com{\nabla^2_A}{\nabla^2_B}$ yields
  \begin{align}
    \label{eq:27}
    &B^{uv} \nabla^2_{uv}A^{xy}\nabla^2_{xy} s = A^{xy} B^{uv}
    \nabla^4_{uvxy} s + 2 B^{uv} \nabla_v (A^{xy}) \nabla^3_{uxy} s +
    B^{uv} \nabla^2_{uv} (A^{xy}) \nabla^2_{xy} s,
  \end{align}
  where the second term can be rewritten as
  \begin{align}
    \label{eq:36}
    \begin{aligned}
      &2 B^{uv} \nabla_v (A^{xy}) \nabla^3_{uxy} s \\&\qquad=
      \frac{2}{3} B^{uv} \nabla_v (A^{xy}) \bigr( \nabla^3_{uxy} + 2
      \nabla^3_{xuy} s - 2 R^r_{uxy} \nabla_r - i 2 k\, \omega_{ux}
      \nabla_y \bigl) s \\&\qquad = \frac{2}{3} B^{uv} \nabla_y
      (A^{xy}) \big( \nabla^3_{uxy} + \nabla^3_{xuy} + \nabla^3_{xyu}
      - 2 R^r_{uxy} \nabla_r - i 3 k\, \omega_{ux} \nabla_y \big) s.
    \end{aligned}
  \end{align}
  By subtracting \eqref{eq:27} from \eqref{eq:13}, substituting
  \eqref{eq:35} and \eqref{eq:36}, and collecting terms by order of
  covariant differentiation, one verifies the claimed symbols.
\end{proof}

\chapter{The Hitchin Connection}
\label{cha:hitchin-connection}

In this section, we study the Hitchin connection and calculate its
curvature. We start by recalling the differential geometric
construction of the Hitchin connection in geometric quantization. The
results concerning this construction are all proved in
\cite{MR2928087}, to which the reader is referred for further details.

Consider a symplectic manifold $(M,\omega)$, equipped with a
prequantum line bundle $\mathcal{L}$, and assume that $H^1(M, \setR) =
0$ and that the real first Chern class of $(M, \omega)$ is given by
\begin{align}
  \label{eq:521}
  c_1(M, \omega) = \ccint \big [ \tfrac{\omega}{2\pi} \big],
\end{align}
for some integer $\ccint \in \setZ$. Further, assume that $M$ is of
K\"ahler type, and let $J$ be a rigid and holomorphic family of
K\"ahler structures on $(M,\omega)$, parametrized by some complex
manifold $\mathcal{T}$. Finally, assume that the family of K\"ahler
structures admits a family of Ricci potentials $F$, and that it does
not admit any non-constant holomorphic functions on $M$.

The prequantum space $\pqs = C^\infty(M, \mathcal{L}^k)$
forms the fiber of a trivial, infinite-rank vector bundle over
$\mathcal{T}$,
\begin{align*}
  \pqb = \mathcal{T} \times \pqs.
\end{align*}
If $\trivcon$ denotes the trivial connection on $\pqb$, we
consider a connection of the form
\begin{align}
  \label{eq:518}
  \boldnabla = \trivcon + a,
\end{align}
where $a \in \Omega^1(\mathcal{T}, \mathcal{D}(M, \mathcal{L}^k))$ is
a one-form on $\mathcal{T}$ with values in the space of differential
operators on sections of $\mathcal{L}^k$, and we seek an $a$ for which
the connection $\boldnabla$ preserves the quantum subspaces
$\qs_\sigma = H^0(M_\sigma, \mathcal{L}^k)$ of holomorphic
sections inside each fiber of $\pqb$.

\begin{definition}
  \label{def:11}
  A \emph{Hitchin connection} on the bundle $\pqb$ is a
  connection of the form \eqref{eq:518} which preserves the fiberwise
  subspaces $\qs$.
\end{definition}

It turns out that, with the assumptions made above, an explicit
construction of a Hitchin connetion can be given. For any vector field
$V$ on $\mathcal{T}$, the operator $a(V)$
is of order two, with principal symbol $G(V)$ and lower order symbols
given in terms of $G(V)$ and the Ricci potential.
The precise statement is contained in the following theorem from
\cite{MR2928087}.
\begin{theorem}
  \label{thm:58}
  Let $(M, \omega)$ be a prequantizable symplectic manifold with
  $H^1(M, \setR) = 0$ and $c_1(M, \omega) = \ccint \big [
  \frac{\omega}{2\pi} \big ]$. Further, let $J$ be a rigid,
  holomorphic family of K\"ahler structures on $M$, parametrized by a
  complex manifold $\mathcal{T}$, admitting a family of Ricci
  potentials $F$ but no non-constant holomorphic functions on
  $M$. Then the expression
  \begin{align*}
    \boldnabla_V = \trivcon_V + \frac{1}{4k+2\ccint}(\Delta_{G(V)} + 2
    \nabla_{G(V) \trdot dF} + 4kV'[F])
  \end{align*}
  defines a Hitchin connection in the bundle $\qb$ over
  $\mathcal{T}$.
\end{theorem}

The characterizing feature of the operator-valued one-form $a$ is the
fact that it satisfies 
\begin{align}
  \label{eq:45}
  \com{\nabla^{0,1}}{a(V)} s = - \smash{\frac{i}{2}} \omega \trdot G(V)
  \trdot \nabla s,
\end{align}
for any section $s$ of $\pqb$. In fact, this property,
and the fact that the Hitchin connection preserves the quantum
subspaces $\qs$ inside $\pqs$, implies that these
subspaces form a bundle $\qb$ over $\mathcal{T}$, and this
is part of the statement in \Fref{thm:58}.

Having reviewed the explicit differential geometric construction of a
Hitchin connection in geometric quantization, we turn to the
calculation of its curvature.

\subsection*{Curvature of the Hitchin Connection}

Suppose the assumptions of \Fref{thm:58} are satisfied, ensuring the
existence of the Hitchin connection, and let us calculate its
curvature. For this calculation, it will be convenient to rewrite the
Hitchin connection slightly as
\begin{align}
  \label{eq:50}
  \boldnabla_V = \trivcon_V + \frac{1}{4k+2\ccint} b(V) + V'[F] \quad
  \text{with} \quad b(V) = \Delta_{G(V)} + 2
  \nabla_{G(V) \trdot dF} - 2\ccint V'[F],
\end{align}
essentially splitting the operator $a(V)$ into orders of $k$. In
particular, the one-form $b$ does not involve the level $k$.

We shall divide the calculation of the curvature into a number of
propositions. The first relies on \Fref{lem:7} and \Fref{lem:8}, in
combination with \Fref{prop:2}, to compute a commutator of fundamental
importance to the curvature calculation.

\begin{proposition}
  \label{prop:3}
  For any rigid family of K\"ahler structures, the commutator of
  $\Delta_{\tilde G(V)}$ and $\Delta_{\tilde G(W)}$, acting on
  $\pqs$, has the symbols
  \begin{align*}
    \sigma_3 \com{\Delta_{\tilde G(V)}}{\Delta_{\tilde G(W)}} &= 0 \\
    \sigma_2 \com{\Delta_{\tilde G(V)}}{\Delta_{\tilde G(W)}} &=
    -4ik\;\!
    \Theta(V,W) \\
    \sigma_1\com{\Delta_{\tilde G(V)}}{\Delta_{\tilde G(W)}} &=
    \Delta_{G(V)} \delta G(W) - \Delta_{G(W)} \delta G(V) +
    \Delta_{\bar
      G(V)} \delta \bar G(W) - \Delta_{\bar G(W)} \delta \bar G(V) \\[-1pt]
    & \qquad \smash{- 4ik\;\! \delta (\Theta(V,W)) - \delta (\tilde
      G(V) \trdot r) \trdot \tilde G(W) + \delta (\tilde G(W) \trdot r
      ) \trdot \tilde G(V)}
    \\[1pt]
    \sigma_0 \com{\Delta_{\tilde G(V)}}{\Delta_{\tilde G(W)}} &= -
    ik\;\!  \delta \delta \Theta (V,W) + ik\;\! r (\Theta(V,W)) ,
  \end{align*}
  for any vector fields $V$ and $W$ on $\mathcal{T}$. 
\end{proposition}

\begin{proof}
  We verify that the third-order symbol vanishes. Using \Fref{prop:2}
  and its conjugated version, which rely heavily on rigidity on the
  family, we get 
  \begin{align*}
    \sigma_3\com{\Delta_{\tilde G(V)}}{\Delta_{\tilde G(W)}} &=
    \sigma_3\com{\nabla^2_{\tilde G(V)}}{\nabla^2_{\tilde G(W)}} \\
    &= 2 \Sym(\tilde G(V) \trdot \nabla \tilde G(W)) - 2 \Sym (\tilde G(W)
    \trdot \nabla \tilde G(V)) \\
    &= 2 \Sym(G(V) \trdot \nabla G(W)) - 2 \Sym (G(W) \trdot \nabla G(V))
    \\ & \qquad + 2 \Sym(\bar G(V) \trdot \nabla \bar G(W)) - 2 \Sym
    (\bar G(W) \trdot \nabla \bar G(V)) \\ &= 2 \Gamma_3(V,W) - 2 \bar
    \Gamma_3(W,V) + 2
    \Gamma_3(V,W) - 2 \bar \Gamma_3(W,V) \\
    &= 0.
  \end{align*}
  To calculate the second-order symbol, we first notice that
  \begin{align*}
    \sigma_2\com{\Delta_{G(V)}}{\Delta_{G(W)}} &= \:
    \sigma_2\com{\nabla^2_{G(V)}}{\nabla^2_{G(W)}} +
    \sigma_2\com{\nabla^2_{G(V)}}{\nabla_{\delta G(W)}} -
    \sigma_2\com{\nabla^2_{G(W)}}{\nabla_{\delta G(V)}} \\ &
    = \nabla^2_{G(V)} G(W) + \nabla_{\delta G(V)} G(W) + 2 \Sym(G(V)
    \trdot \nabla \delta G(W))     \\
    & \qquad - \nabla^2_{G(W)} G(V) -
    \nabla_{\delta G(W)} G(V) - 2 \Sym(G(W) \trdot \nabla \delta G(V)) \\
    &=\Gamma_2(V,W) - \Gamma_2(W,V) \\ &= 0.
  \end{align*}
  For the mixed-type terms of the second-order symbol, we first observe that
  \begin{align}
    \label{eq:2}
    \begin{aligned}
      G(V) \trdot \nabla \delta \bar G(W) &= G(V)^{xy} \nabla_y \nabla_u
      \bar G(W)^{uv} \\ &= G(V)^{xy} R^r_{yru} \bar G(W)^{uv} +
      G(V)^{xy} R^r_{yuv} \bar G(W)^{uv} \\ &= - G(V) \trdot r \trdot
      \bar G(W) + G(V)^{xy} R^r_{yuv} \bar G(W)^{uv}.
    \end{aligned}
  \end{align}
  Using this, we calculate that
  \begin{align*}
    &\sigma_2\com{\Delta_{G(V)}}{\Delta_{\bar G(W)}} \\ & \qquad =
    \: \sigma_2 \com{\nabla^2_{G(V)}}{\nabla^2_{\bar G(W)}} +
    \sigma_2\com{\nabla^2_{ G(V)}}{\nabla_{ \delta \bar G(W)}} -
    \sigma_2\com{\nabla^2_{\bar G(W)}}{\nabla_{\delta G(V)}} \\
    &\qquad = - 4ik \Sym(G(V) \trdot \omega \trdot \bar G(W)) + 2
    \Sym(G(V)^{xy} R^a_{uxy} \bar G(W)^{ub}) - 2 \Sym(\bar G(W)^{uv}
    R^a_{xuv} G(V)^{xb}) \\ & \qquad \qquad + 2 \Sym(G(V) \trdot \nabla
    \delta \bar G(W)) - 2 \Sym(\bar G(W)
    \trdot \nabla \delta G(V)) \\
    &\qquad = - 4ik \Sym(G(V) \trdot \omega \trdot \bar G(W)),
  \end{align*}
  where the last equality follows by inserting \eqref{eq:2} and its
  conjugate. In total, this means that the second-order symbol is
  given by
  \begin{align*}
    \sigma_2 \com{\Delta_{\tilde G(V)}}{\Delta_{\tilde G(W)}} &= - 4ik
    \Sym(G(V) \trdot \omega \trdot \bar G(W)) - 4ik \Sym(\bar G(V)
    \trdot \omega \trdot G(W)) \\ &= - 4ik \Sym(\tilde G(V) \trdot \omega
    \trdot \tilde G(W)) \\ & = - 4ik\, \Theta(V,W). \phantom{\tilde G(V)}
  \end{align*}
  Next, we find the first-order symbol. First of all, we get
  \begin{align*}
    \sigma_1\com{\Delta_{G(V)}}{\Delta_{G(W)}} &= \Delta_{G(V)}
    \delta G(W) - \Delta_{G(W)} \delta G(V).
  \end{align*}
  \noindent 
  For the terms of mixed type on $M$, we get the following computation,
  \begin{align*}
    &\sigma_1\com{\Delta_{G(V)}}{\Delta_{\bar G(W)}}
    \\
    & \quad = \sigma_1 \!\!\:\com{\nabla^2_{G(V)}}{\!\nabla^2_{\bar
        G(W)}} \!  + \sigma_1 \!\!\: \com{\nabla_{\delta
        G(V)}}{\!\nabla_{\delta \bar G(W)}} \! + \sigma_1 \!\!\:
    \com{\nabla^2_{G(V)}}{\!\nabla_{\delta \bar G(W)}} \!- \sigma_1
    \!\!\: \com{\nabla^2_{\bar G(W)}}{\!\nabla_{\delta G(V)}}
    \\
    & \quad = - \hcancel{2ik G(V)^{xy} \omega_{yu} \nabla_x (\bar
      G(W)^{ua})} + \hcancel{2ik \bar G(W)^{uv} \omega_{vx} \nabla_u
      (G(V)^{xa})}
    \\
    & \qquad \qquad - G(V)^{xy} \nabla_x (R^a_{yuv}) \bar G(W)^{uv} +
    \bar G(W)^{uv} \nabla_v (R^a_{uxy}) G(V)^{xy}
    \\
    & \qquad \qquad- \hcancel{\frac{4}{3} G(V)^{xy} R^a_{xuv} \nabla_y
      \bar G(W)^{uv}} + \hcancel{\frac{4}{3} \bar G(W)^{uv} R^a_{uxy}
      \nabla_v G(V)^{xy}}
    \\
    & \qquad\;\! - \delta G(V) \trdot r \trdot \bar G(W) +
    \hcancel{\delta G(V)^{x} R^a_{xuv} \bar G(W)^{uv}}
    \\
    & \qquad \qquad + \delta \bar G(W) \trdot r \trdot G(V) -
    \hcancel[0.2pt]{ \delta \bar G(W)^{u} R^a_{uxy} G(V)^{xy}}
    \\
    & \qquad\;\! + \nabla^2_{G(V)} \delta \bar G(W) - 2ik G(V) \trdot
    \omega \trdot \delta \bar G(W) + \hcancel{G(V)^{xy} R^a_{uxy}
      \delta \bar G(W)^u}
    \\
    & \qquad\;\! - \nabla^2_{\bar G(W)} \delta G(V) + 2ik \bar G(W)
    \trdot \omega \trdot \delta G(V) - \hcancel{\bar G(W)^{uv}
      R^a_{xuv} \delta G(V)^x}
    \\
    & \quad = - \hcancel{G(V)^{xy} \nabla_x (R^a_{yuv}) \bar
      G(W)^{uv}} + \hcancel{\bar G(W)^{uv} \nabla_v (R^a_{uxy})
      G(V)^{xy}}
    \\
    & \qquad\;\! - \delta G(V) \trdot r \trdot \bar G(W) + \delta \bar
    G(W) \trdot r \trdot G(V)
    \\
    & \qquad\;\!- G(V)^{xy} \nabla_x (r_{yu}) \bar G(W)^{ua} +
    \hcancel{G(V)^{xy} \nabla_x (R^a_{yuv}) \bar G(W)^{uv}} - 2ik G(V)
    \trdot \omega \trdot \delta \bar G(W)
    \\
    & \qquad\;\! + \bar G(W)^{uv} \nabla_u (r_{vx}) G(V)^{xa} -
    \hcancel{\bar G(V)^{uv} \nabla_u (R^a_{vxy}) G(V)^{xy}} + 2ik \bar
    G(W) \trdot \omega \trdot \delta G(V)
    \\
    & \quad = - \delta G(V) \trdot r \trdot \bar G(W) + \delta \bar
    G(W) \trdot r \trdot G(V)
    \\
    & \qquad\;\! - G(V)^{xy} \nabla_x (r_{yu}) \bar G(W)^{ua} - 2ik
    G(V) \trdot \omega \trdot \delta \bar G(W)
    \\
    & \qquad\;\! + \bar G(W)^{uv} \nabla_u (r_{vx}) G(V)^{xa} + 2ik
    \bar G(W) \trdot \omega \trdot \delta G(V)
    \\
    & \quad = - \delta (G(V) \trdot r) \trdot \bar G(W) \! + \!\delta
    (\bar G(W) \trdot r ) \trdot G(V)\! + \!2ik \delta (\bar G(W)
    \trdot \omega \trdot G(V)) \!- \! 2ik \delta (G(V) \trdot \omega
    \trdot \bar G(W))
    \\
    & \quad = - \delta (G(V) \trdot r) \trdot \bar G(W) + \delta (\bar
    G(W) \trdot r ) \trdot G(V) - 4ik \delta \Sym(G(V) \trdot \omega
    \trdot \bar G(W)),
  \end{align*}
  where we have indicated cancelling terms. Analagously, we compute
  \begin{align*}
    &\sigma_1\com{\Delta_{\bar G(V)}}{\Delta_{G(W)}} = \delta (G(W)
    \trdot r) \trdot \bar G(V) - \delta (\bar G(V) \trdot r ) \trdot
    G(W) - 4ik \delta \Sym(\bar G(V) \trdot \omega \trdot G(W)),
  \end{align*}
  so that finally
  \begin{align*}
    \sigma_1\com{\Delta_{\tilde G(V)}}{\Delta_{\tilde G(W)}} &= - 4i
    \delta (\Theta(V,W)) - \delta (\tilde G(V) \trdot r) \trdot \tilde
    G(W) + \delta (\tilde G(W) \trdot r ) \trdot \tilde G(V) \\ &
    \quad \;\! + \Delta_{G(V)} \delta G(W) - \Delta_{G(W)} \delta G(V)
    + \Delta_{\bar G(V)} \delta \bar G(W) - \Delta_{\bar G(W)} \delta
    \bar G(V).
  \end{align*}
  For the symbol of order zero, we first observe that
  \begin{align*}
    &\sigma_0\com{\Delta_{G(V)}}{\Delta_{G(W)}} = 0.
  \end{align*}
  For the mixed terms, we get
  \begin{align*}
    &\sigma_0\com{\Delta_{G(V)}}{\Delta_{\bar G(W)}} \\ &\quad =
    \sigma_0\com{\nabla^2_{G(V)}}{\!\nabla^2_{\bar G(W)}} \!+
    \sigma_0\com{\nabla_{\delta G(V)}}{\!\nabla_{\delta \bar G(W)}}
    \!+ \sigma_0 \com{\nabla^2_{G(V)}}{\!\nabla_{\delta \bar G(W)}}\!
    - \sigma_0 \com{\nabla^2_{\bar G(W)}}{\!\nabla_{\delta G(V)}} \\
    &\quad = - kG(V)^{xy} R_{xuvy} \bar G(W)^{uv} - ik \delta G(V)
    \trdot \omega \trdot \delta \bar G(W) \\ & \qquad\;\! - ik
    \omega(G(V) \trdot \nabla \delta \bar G(W)) + ik \omega (\bar G(W)
    \trdot \nabla \delta G(V)) \\ &\quad = - kG(V)^{xy} R_{xuvy} \bar
    G(W)^{uv} - ik \delta \delta (G(V) \trdot \omega \trdot \bar G(W))
    - ik \omega(G(V) \trdot \nabla \delta \bar G(W)) \\ & \quad =
    -ik\;\! \delta \delta (G(V) \trdot \omega \trdot \bar G(W)) +
    ik\;\! r (G(V) \trdot \omega \trdot \bar G(W)),
  \end{align*}
  where \eqref{eq:2} was applied for the last equation, and this
  finally gives
  \begin{align*}
    \sigma_0\com{\Delta_{\tilde G(V)}}{\Delta_{\tilde G(W)}} = -
    ik\;\! \delta \delta \Theta (V,W) + ik\;\! r( \Theta(V,W)).
  \end{align*}
  This proves the proposition.
\end{proof}

Before stating the next proposition, we introduce a one-form $\mc \in
\Omega^{1,0}(\mathcal{T}, C^\infty(M))$, with values in smooth
functions on $M$, which will play a central role in the
calculations. It is defined by the expression,
\begin{align}
  \label{eq:12}
  \mc(V) = - \Delta_{G(V)}F - d F \trdot G(V) \trdot d F - 2\ccint
  V'[F].
\end{align}
This one-form serves, in fact, as the zero-order part of the Hitchin
connection in metaplectic quantization, studied in \cite{MR2928088},
but as we will see, it also appears in the curvature of the Hitchin
connection in the setting considered here. In the metaplectic case,
the crucial property satisfied by this one-form is the following
relation,
\begin{align}
  \label{eq:34}
  \bar \d \mc(V) = \frac{i}{2} \delta(G(V) \trdot \rho),
\end{align}
which will also be useful to us.  This is verified through the
calculation
\begin{align*}
  2 \bar \d \Delta_{G(V)} F &= 2 \bar \d \delta (G(V) \trdot dF) \\
  &= - 2i \rho \trdot G(V) \trdot dF + 2\delta (G(V) \trdot \d \bar \d
  F) \\ &= -2i \ccint \omega \trdot G(V) \trdot dF + 4 \d \bar \d F
  \trdot G(V) \trdot dF - i \delta( G(V) \trdot \rho) + i \ccint
  \delta G(V) \trdot \omega \\ &= -4 \ccint \bar \d V'[F] - 2\bar \d
  (dF \trdot G(V) \trdot dF) - i\delta (G(V) \trdot \rho),
\end{align*}
where we applied \eqref{eq:23} twice for the third equality and
\eqref{eq:41} for the last equality.

As hinted above, the exterior derivative of the one-form $\mc$ over
$\mathcal{T}$ appears in the curvature of the Hitchin connection. To
see how, we must be able to recognize this derivative.
\begin{proposition}
  \label{prop:24}
  The exterior derivative of the one-form $\mc$ defined by
  \eqref{eq:12} is given by
  \begin{align*}
    \d_{\scriptscriptstyle \mathcal{T}}\mc(V,W) &= b(V)
    W'[F] - b(W) V'[F] \\
    \bar \d_{\scriptscriptstyle \mathcal{T}} \mc(V,W) &= \frac{i}{4}
    \delta\delta \Theta(V,W) - \frac{i}{4} r(\Theta (V,W)) - i \ccint
    \theta(V, W) - 2\ccint \d_{\scriptscriptstyle \mathcal{T}} \bar
    \d_{\scriptscriptstyle \mathcal{T}} F (V,W)
  \end{align*}
  for any vector fields $V$ and $W$ on $\mathcal{T}$.
\end{proposition}

\begin{proof}
  For the first statement, choose $V$ and $W$ so that $V'$ and $W'$
  commute. Using the identity $\Delta_{G(V)} F = \delta (G(V) \trdot
  dF)$ and the fact that the divergence operator on vector fields does
  not depend on the K\"ahler structure, we get
  \begin{align*}
    \d_{\scriptscriptstyle \mathcal{T}} \mc(V,W) &= V'[\mc (W)] -
    W'[\mc(V)] \\ &= - \Delta_{V'[G(W)]} F - \Delta_{G(W)} V'[F] - 2
    dF \trdot G(W) \trdot dV'[F] \\ & \qquad + \Delta_{W'[G(V)]} F +
    \Delta_{G(V)} W'[F] + 2 dF \trdot G(V) \trdot dW'[F] \\ &= b(V)
    W'[F] - b(W) V'[F],
  \end{align*}
  where we used the the fact that $W'[G(V)] = - W'V'[\tilde g] =
  V'[G(W)]$, since the vector fields were chosen to commute.

  \pagebreak[1] For the second statement, choose $V$ and $W$ so that
  $V'$ and $W''$ commute. Once again using $\Delta_{G(V)} F = \delta
  (G(V) \trdot dF)$, the identities \eqref{eq:43} and \eqref{eq:41}
  give us that
  \begin{align*}
    &4W''[\Delta_{G(V)}(F)] \\ &\qquad= 4\delta(W''[G(V)] \trdot dF) +
    4\delta( G(V) \trdot W''[F]) \\ & \qquad= i \delta (\delta \bar
    G(W) \trdot \omega \trdot G(V)) - 2i \delta (\bar G(W) \trdot
    \omega \trdot G(V) \trdot dF) \\ &\qquad= i \delta \delta (\bar
    G(W) \trdot \omega \trdot G(V)) - 2i \delta \bar G(W) \trdot
    \omega \trdot G(V) \trdot dF + 2i \d \bar \d F({\bar G(W) \trdot
      \omega \trdot G(V)}),
  \end{align*}
  where rigidity of the family of K\"ahler structures was used for the
  last equality. The Ricci potential satisfies the equation $\rho =
  \ccint \omega + 2i\d \bar \d F$, so we get that
  \begin{align*}
    2i\d \bar \d F({\bar G(W) \trdot \omega \trdot G(V)}) &= \rho(\bar
    G(W) \trdot \omega \trdot G(V)) - \ccint \omega (\bar G(W) \trdot
    \omega \trdot G(V) ) \\ &= ir(\Theta(V',W'')) + 4i\ccint
    \theta(V', W'').
  \end{align*}
  Finally, the identities \eqref{eq:43} and \eqref{eq:41} can be used
  to verify that
  \begin{align*}
    4W''[dF \trdot G(V) \trdot dF] = 2i \delta \bar G(W) \trdot \omega
    \trdot G(V) \trdot dF.
  \end{align*}
  Combining the identities above, we get
  \begin{align*}
    &4\,\bar \d_{\scriptscriptstyle \mathcal{T}} \mc(V',W'') \\ &
    \qquad = - 4W''[\Delta_{G(V)}F + d F \trdot G(V) \trdot d F +
    2nV'[F]] \\ & \qquad = i \delta\delta \Theta(V',W'') -
    ir(\Theta(V',W'')) - 4i\ccint \theta(V', W'') - 8\ccint
    \d_{\scriptscriptstyle \mathcal{T}}\bar \d_{\scriptscriptstyle
      \mathcal{T}} F(V', W'').
  \end{align*}
  For the first term of the last equality, we used the fact that
  repeated application of the divergence operator to a bivector field
  only depends on its symmetric part. This finishes the proof of the
  proposition.
\end{proof}

We will also need the following
\begin{proposition}
  \label{prop:8}
  The one-form $\mc$ defined in \eqref{eq:12} satisifes
  \begin{align*}
    X'_{\d_{\scale{0.7}{\scriptscriptstyle \mathcal{T}}} \mc (V,W)} =
    \frac{i}{4} \Delta_{G(W)} \delta G(V) - \frac{i}{4} \Delta_{G(V)}
    \delta G(W) + \frac{i}{2} G(W) \trdot d \mc(V) - \frac{i}{2} G(V)
    \trdot d \mc(W)
  \end{align*}
  for any vector fields $V$ and $W$ on $\mathcal{T}$.
\end{proposition}

\begin{proof}
  As usual, the Hamiltonian vector field in the statement is
  determined by
  \begin{align*}
    X'_{\d_{\scale{0.7}{\scriptscriptstyle \mathcal{T}}} \mc (V,W)}
    \trdot \omega = \bar \d (\d_{\scriptscriptstyle \mathcal{T}} \mc
    (V,W)).
  \end{align*}
  To calculate the right-hand side of this, we use the fact that $\mc$
  satisfies \eqref{eq:34} to get
  \begin{align*}
    \frac{i}{2} W'[\delta (G(V) \trdot \rho)] = W'[ \bar \d \mc(V)] =
    \bar \d W'[\mc(V)] - \frac{i}{2} \omega \trdot G(W) \trdot d
    \mc(V).
  \end{align*}

  \pagebreak[1] 
  \noindent On the other hand, we calculate
  \begin{align*}
    W'[\delta (G(V) \trdot \rho)] &= W'[\nabla_x( G(V)^{xy}
    \rho_{xa})] \\ &= W'[\lcc]^z_{zx} G(V)^{xy} \rho_{xa} -
    W'[\lcc]^z_{xa} G(V)^{xy}\rho_{yz} + \delta (W'[G(V) \trdot \rho])
    \\ &= \delta (W'[G(V)] \trdot \rho) + \delta (G(V) \trdot
    W'[\rho]),
  \end{align*}
  where the last equality uses \eqref{eq:8} and \eqref{eq:48} and type
  considerations. Applying \eqref{eq:1234}, we get
  \begin{align*}
    2\delta (G(V) \trdot W'[\rho]) &= \delta (G(V) \trdot d (\delta
    G(W) \trdot \omega)) = \delta (G(V) \trdot \nabla \delta G(W))
    \trdot \omega = \Delta_{G(V)} \delta G(W) \trdot \omega.
  \end{align*}
  Combining the previous three identities, we find
  \begin{align*}
    &\bar \d \d_{\scriptscriptstyle \mathcal{T}} \mc(V,W) = \bar \d
    V'[\mc(W)] - \bar \d W'[\mc(V)] \\ & \qquad = \frac{i}{2} \omega
    \trdot G(V) \trdot d \mc(W) - \frac{i}{2} \omega \trdot G(W)
    \trdot d \mc(V) + \frac{i}{4} \Delta_{G(W)} \delta G(V) \trdot
    \omega - \frac{i}{4} \Delta_{G(V)} \delta G(W) \trdot \omega,
  \end{align*}
  for commuting vector fields $V'$ and $W''$.  Raising the index with
  $\omega$ ends the proof.
\end{proof}

The curvature calculation for the Hitchin connection proceeds with
expressions for commutators involving the one-form $b$ defined in
\eqref{eq:50}.

\begin{proposition}
  \label{prop:22}
  The commutator of the operators $b(V)$ and $b(W)$, acting on
  sections of $\pqb$, is a first-order operator with symbols given by
  \begin{align*}
    \sigma_1\com{b(V)}{b(W)} &= 4i
    X'_{\d_{\scale{0.7}{\scriptscriptstyle \mathcal{T}}} \mc (V,W)}
    \\
    \sigma_0 \com{b(V)}{b(W)} &= -2\ccint \d_{\scriptscriptstyle
      \mathcal{T}} c(V,W),
  \end{align*}
  for any vector fields $V$ and $W$ on $\mathcal{T}$.
\end{proposition}

\begin{proof}
  The vanishing of the third-order symbol is essentially
  \Fref{prop:3},
  \begin{align*}
    \sigma_3\com{b(V)}{b(W)} = \sigma_3
    \com{\Delta_{G(V)}}{\Delta_{G(W)}} = -4ik\Theta(V',W') = 0.
  \end{align*}
  Vanishing of the second-order symbol is seen through the following
  calculation using \Fref{lem:7} and \Fref{lem:8},
  \begin{align*}
    &\sigma_2\com{b(V)}{b(W)} \\ & \qquad =
    \sigma_2\com{\Delta_{G(V)}}{\Delta_{G(W)}} + 2
    \sigma_2\com{\Delta_{G(V)}}{\nabla_{G(W)\trdot dF}} - 2
    \sigma_2\com {\Delta_{G(W)}}{\nabla_{G(V) \trdot dF}} \\ & \qquad
    = 2\sigma_2 \com{\nabla^2_{G(V)}}{\nabla_{G(W)\trdot dF}} - 2
    \sigma_2 \com{\nabla^2_{G(W)}}{\nabla_{G(V)\trdot dF}} \\ & \qquad
    = 4 \Sym(G(V) \trdot \nabla (G(W)\trdot dF)) - 4 \Sym(G(W) \trdot
    \nabla (G(V)\trdot dF)) \\ & \qquad \qquad - 2 dF \trdot G(W)
    \trdot \nabla G(V) + 2 dF \trdot G(V) \trdot \nabla G(W) \\ &
    \qquad = 6 \Sym(G(V) \trdot \nabla G(W)) \trdot dF - 6 \Sym(G(W)
    \trdot \nabla G(V)) \trdot dF \\ &
    \qquad = 6 \Gamma_3(V,W) \trdot dF - 6 \Gamma_3(W,V) \trdot dF \\
    & \qquad = 0.
  \end{align*}

  The first-order symbol is the first non-vanishing part. We split its
  calculation in two by first calculating
  \begin{align*}
    &\sigma_1\com{\Delta_{G(V)}}{\nabla_{G(W)\trdot dF}} -
    \sigma_1\com {\Delta_{G(W)}}{\nabla_{G(V)\trdot dF}} \\ & \qquad =
    \sigma_1\com{\nabla^2_{G(V)}}{\nabla_{G(W)\trdot dF}} -
    \sigma_1\com{\nabla^2_{G(W)}}{\nabla_{G(V)\trdot dF}} \\ &\qquad
    \qquad + \sigma_1\com{\nabla_{\delta G(V)}}{\nabla_{G(W)\trdot
        dF}} - \sigma_1\com{\nabla_{\delta G(W)}}{\nabla_{G(V)\trdot
        dF}} \\ & \qquad = \nabla^2_{G(V)} (G(W)\trdot dF) -
    \nabla^2_{G(W)} (G(V)\trdot dF) + [\delta G(V), G(W) \trdot dF] -
    [\delta G(W), G(V) \trdot dF] \\ & \qquad = \Delta_{G(V)}
    (G(W)\trdot dF) - \Delta_{G(W)} (G(V)\trdot dF) - dF \trdot G(W)
    \trdot \nabla \delta G(V) + dF \trdot G(V) \trdot \nabla \delta
    G(W) \\ & \qquad = dF \trdot G(V) \trdot \nabla \delta G(W) - dF
    \trdot G(W) \trdot \nabla \delta G(V) \\ & \qquad \qquad +
    \Delta_{G(V)} (G(W) )\trdot dF + 2 G(V)^{xy} \nabla_y G(W)^{au}
    \nabla^2_{xu} F \\ & \qquad \qquad \qquad + G(W) \trdot d
    \Delta_{G(V)} F - G(W)^{au} \nabla_u (G(V)^{xy}) \nabla^2_{xy}F -
    G(W) \trdot \nabla \delta G(V) \trdot dF \\ & \qquad \qquad -
    \Delta_{G(W)} (G(V) )\trdot dF - 2 G(W)^{uv} \nabla_v G(V)^{ax}
    \nabla^2_{xu} F \\ & \qquad \qquad \qquad - G(V) \trdot d
    \Delta_{G(W)} F + G(V)^{ax} \nabla_x (G(W)^{uv}) \nabla^2_{uv}F +
    G(V) \trdot \nabla \delta G(W) \trdot dF \\ & \qquad = 3
    \Gamma_3(V,W)^{axy} \nabla^2_{xy} F - 3 \Gamma_3(W,V)^{axy}
    \nabla^2_{xy} F + \Gamma_2(V,W) \trdot dF - \Gamma_2(W,V) \trdot
    dF \\ & \qquad \qquad + G(W) \trdot d \Delta_{G(V)} F - G(V)
    \trdot d \Delta_{G(W)} F \\ & \qquad = G(W) \trdot d \Delta_{G(V)}
    F - G(V) \trdot d \Delta_{G(W)} F.
  \end{align*}
  Here \Fref{prop:2}, once again, played a central role. Similarly, we
  find that
  \begin{align*}
    2\sigma_1 \com{\nabla_{G(V)\trdot dF}}{\nabla_{G(W)\trdot dF}} & =
    2\com{G(V)\trdot dF}{G(W)\trdot dF} \\ & = 3 dF \trdot
    \Gamma_3(V,W) \trdot dF - 3 dF \trdot \Gamma_3(W,V) \trdot dF \\ &
    \qquad + G(W) \trdot d( dF
    \trdot G(V) \trdot dF) - G(V) \trdot d( dF \trdot G(W) \trdot dF) \\
    & = G(W) \trdot d( dF \trdot G(V) \trdot dF) - G(V) \trdot d( dF
    \trdot G(W) \trdot dF).
  \end{align*}
  Combining the last two computations, we see that
  \begin{align*}
    \sigma_1 \com{b(V)}{b(W)} &= \Delta_{G(V)} \delta G(W) -
    \Delta_{G(V)} \delta G(W) \\ & \qquad + 2G(W) \trdot d
    \Delta_{G(V)} F - 2 G(V) \trdot d \Delta_{G(W)} F \\ & \qquad + 2
    G(W) \trdot d( dF \trdot G(V) \trdot dF) - 2 G(V) \trdot d( dF
    \trdot G(W) \trdot dF) \\ & \qquad + 4\ccint G(W) \trdot d V'[F] -
    4\ccint G(V) \trdot d W'[F] \\ &= \Delta_{G(V)} \delta G(W) -
    \Delta_{G(V)} \delta G(W) \\ & \qquad + 2 G(V) \trdot d \mc(W) - 2
    G(W) \trdot d \mc(V),
  \end{align*}
  where the last equation follows by recalling \eqref{eq:12}.
  
  Finally, for the zeroth-order symbol, we rely on \Fref{prop:24} to
  get that
  \begin{align*}
    \sigma_0 \com{b(V)}{b(W)} &= 2\ccint b(W) V'[F] - 2\ccint b(V)
    W'[F] = -2\ccint \d_{\scriptscriptstyle \mathcal{T}} \mc(V,W).
  \end{align*}
  This finishes the proof of the proposition.
\end{proof}

\pagebreak[2] Finally, to calculate the curvature of the Hitchin
connection, we need the extorior derivative of the one-form $b$ over
$\mathcal{T}$. This is calculated in the following proposition.

\begin{proposition}
  \label{prop:11}
  The two-form $d_{\scriptscriptstyle \mathcal{T}}b \in
  \Omega^2(\mathcal{T}, \mathcal{D}(M, \mathcal{L}^k))$ is given by
  \begin{align*}
    d_{\scriptscriptstyle \mathcal{T}} b(V, W) &= - i
    \Delta_{\Theta(V,W)} - 2i\nabla_{\Theta(V,W) \trdot dF} - 2
    \nabla_{G(V)\trdot dW[F] - G(W) \trdot dV[F]} + 2 \ccint
    \d_{\scriptscriptstyle \mathcal{T}} \bar \d_{\scriptscriptstyle
      \mathcal{T}} F(V,W)
  \end{align*}
  on sections of $\pqb$, and by
  \begin{align*}
    d_{\scriptscriptstyle \mathcal{T}} b(V, W) &= 2 \nabla_{G(W)
      \trdot V'[F]} - 2 \nabla_{G(V) \trdot W'[F]} - 2ik \theta(V,W) +
    2\ccint \d_\mathcal{T} \bar \d_{\scriptscriptstyle \mathcal{T}}
    F(V, W)
  \end{align*}
  when restricted to sections of $\qb$.
\end{proposition}

\begin{proof}
  For any section $s$ of $\pqb$ over $\mathcal{T}$, we have the
  identity
  \begin{align}
    \label{eq:6}
    \begin{aligned}
      W[\Delta_{G(V)}]s &= W[\nabla_x G(V)^{xy} \nabla_y]s \\ &=
      W[\lcc]^z_{zx} G(V)^{xy} \nabla_y + \nabla_x W[G(V)]^{xy}
      \nabla_y s \\ &= \Delta_{W[G(V)]} s,
    \end{aligned}
  \end{align}
  where the last equality follows from \eqref{eq:48}. In particular,
  for commuting $V$ and $W$, we get
  \begin{align*}
    V[\Delta_{G(W)}] - W[\Delta_{G(V)}] = \Delta_{V[G(W)]} -
    \Delta_{W[G(V)]} = \Delta_{d_{\scriptscriptstyle \mathcal{T}}
      G(V,W)} = -i \Delta_{\Theta(V,W)},
  \end{align*}
  where \eqref{eq:44} was used for the last equation.  Similarly, we
  calculate
  \begin{align*}
    W[\nabla_{G(V) \trdot dF}] s &= \nabla_{G(V)\trdot W[F]} s +
    \nabla_{W[G(V)] \trdot dF} s,
  \end{align*}
  so that finally
  \begin{align}
    \label{eq:46}
    & d_{\scriptscriptstyle \mathcal{T}} b(V, W) = V[b(W)] - W[b(V)]
    \notag \\ & \qquad = -i \Delta_{\Theta(V,W)} -
    2i\nabla_{\Theta(V,W) \trdot dF} - 2\nabla_{G(V)\trdot dW[F] -
      G(W) \trdot dV[F]} + 2\ccint \d_{\scriptscriptstyle \mathcal{T}}
    \bar \d_{\scriptscriptstyle \mathcal{T}} F(V,W).
  \end{align}
  This proves the first statement of the proposition.

  For the second statement, observe that the stated formulas on
  sections of $\pqb$ and $\qb$ agree for $\d_{\scriptscriptstyle
    \mathcal{T}} b$. For the case of $\bar \d_{\scriptscriptstyle
    \mathcal{T}} b$, suppose that $s$ is a section of $\qb$ and
  observe that
  \begin{align*}
    2 \nabla^2_{\Theta(V',W'')} s = - \nabla^2_{\bar G(W) \trdot
      \omega \trdot G(V)} s = ik \omega (\bar G(W) \trdot \omega
    \trdot G(V)) s = 4k \theta(V',W'') s.
  \end{align*}
  Using the conjugated version of \eqref{eq:41}, we also get that
  \begin{align*}
    2\nabla_{\delta \Theta(V',W'')} s + 4\nabla_{\Theta(V',W'') \trdot
      dF} s &= \nabla_{ G(V) \trdot \omega \trdot \delta \bar G(W)} s
    + 2\nabla_{G(V)\trdot \omega \trdot \bar G(W) \trdot dF} s \\ &=
    4i\nabla_{G(V) \trdot dW''[F]} s.
  \end{align*}
  Finally, plugging these two expressions into \eqref{eq:46}, we get
  \begin{align*}
    d_{\scriptscriptstyle \mathcal{T}} b(V', W'')s &= - 2 \nabla_{G(V)
      \trdot W'[F]} s - 2ik \theta(V',W'') s + 2\ccint \d_\mathcal{T}
    \bar \d_{\scriptscriptstyle \mathcal{T}} F(V', W'') s,
  \end{align*}
  which finishes the proof of the proposition.
\end{proof}

With the above propositions at hand, calculating the curvature of the
Hitchin connection is a straightforward matter.

\begin{theorem}
  \label{thm:4}
  The curvature of the Hitchin connection acts by
  \begin{align*}
    \hcc^{2,0} = \frac{k}{(2k+\ccint)^2} P_k(\d_{\scriptscriptstyle
      \mathcal{T}} \mc) \qquad \quad \hcc^{1,1} = -
    \frac{ik}{2k+\ccint} (\theta - 2i \d_{\scriptscriptstyle
      \mathcal{T}} \bar \d_{\scriptscriptstyle \mathcal{T}} F) \qquad
    \quad \hcc^{0,2} = 0,
  \end{align*}
  on sections of the bundle $\qb$.
\end{theorem}

\begin{proof}
  For commuting vector fields $V$ and $W$, we get
  \begin{align*}
    &\hcc(V,W) = \com{\boldnabla_V}{\boldnabla_W} \\
    & \qquad = \frac{\com{b(V)}{b(W)}}{(4k+2\ccint)^2} +
    \frac{d_{\scriptscriptstyle \mathcal{T}}b(V,W) + \com{b(V)}{W'[F]}
      - \com{b(W)}{V'[F]}}{4k+2\ccint} - \d_{\scriptscriptstyle
      \mathcal{T}}\bar \d_{\scriptscriptstyle \mathcal{T}} F(V,W).
  \end{align*}
  Now \Fref{prop:24} can be applied to see that
  \begin{align*}
    \com{b(V)}{W'[F]} - \com{b(W)}{V'[F]} = 2\nabla_{G(V) \trdot
      W'[F]} - 2\nabla_{G(W) \trdot V'[F]} + \d_{\scriptscriptstyle
      \mathcal{T}}\mc(V,W),
  \end{align*}
  and combining the above with \Fref{prop:22} and \Fref{prop:11}
  yields the following expression for the (2,0)-part of the curvature
  \begin{align*}
    \hcc^{2,0}(V,W) &= \frac{\com{b(V)}{b(W)}}{(4k+2\ccint)^2} +
    \frac{\d_{\scriptscriptstyle \mathcal{T}}\mc(V,W)}{4k+2\ccint} =
    \frac{k}{(2k+\ccint)^2} P_k(\d_{\scriptscriptstyle
      \mathcal{T}}\mc(V,W)),
  \end{align*}
  where $P_k$ denotes the prequantum operator defined in \eqref{eq:9}.
  For the (1,1)-part of the curvature, \Fref{prop:11} yields
  \begin{align*}
    &\hcc^{1,1}(V,W) = \frac{\bar \d_{\scriptscriptstyle \mathcal{T}}
      b(V,W)}{4k+2\ccint} - \d_{\scriptscriptstyle \mathcal{T}}\bar
    \d_{\scriptscriptstyle \mathcal{T}} F(V,W) = -
    \frac{ik}{2k+\ccint} (\theta - 2i \d_{\scriptscriptstyle
      \mathcal{T}} \bar \d_{\scriptscriptstyle \mathcal{T}} F).
  \end{align*}
  Finally, the (0,2)-part of the curvature clearly vanishes, and the
  theorem is proved.
\end{proof}

The fact that the (1,1)-part of the curvature is an zeroth-order
operator, combined with the fact that the Hitchin connection preserves
the subbundle of quantum spaces $\qb$ implies that the (1,1)-part must
take values in holomorphic and hence constant functions on $M$. This
was, however, already known from \Fref{prop:1} for the particular
expression we found in \Fref{thm:4}.  Also, being a first-order
operator, it is no surprise that the (2,0)-part of the curvature acts
by prequantum operator. After all, prequantum operators of the form
\eqref{eq:7} are the only first-order operators which preserve the
holomorphic sections.

\subsection*{The Hitchin-Witten Connection}

In the previous section, we saw how the higher-order symbols of the
curvature of the Hitchin connection vanished and left a first-order
operator with a relatively simple expression. For the (2,0)-part, the
higher order symbols vanished for general reasons, related to rigidity
of the family of K\"ahler structures, but for the (1,1)-part, the
curvature reduced to a zero-order operator only when restricted the
subbundle $\qb$ of quantum spaces.

It turns out that, with a slight modification of the formula, we can
achieve a cancellation of higher-order terms for the (1,1)-part of the
connection defined on the whole bundle $\pqb$ of prequantum spaces
over $\mathcal{T}$, but still maintain the vanishing of higher-order
symbols for the (2,0) and (0,2)-parts.

In this section, we let $t \in \setC$ be any complex number with
integer real part, $k = \re(t) \in \setZ$, and we consider the
connection on $\pqb$ given by
\begin{align}
  \label{eq:38}
  \tilde \boldnabla_V = \trivcon_V + \frac{1}{2t} b(V) -
  \frac{1}{2\bar t} \bar b(V) + V[F],
\end{align}
where $\bar b (V)$ has the conjugated symbols of $b(V)$ so that
\begin{align*}
  b(V) = \Delta_{ G(V)} + 2 \nabla_{ G(V) \trdot dF} - 2\ccint V'[F]
  \quad \text{and} \quad \bar b(V) = \Delta_{\bar G(V)} + 2 \nabla_{
    \bar G(V) \trdot dF} - 2\ccint V''[F].
\end{align*}
We will refer to $\tilde \boldnabla$ as the Hitchin-Witten
connection. It is a generalization of the connection for quantum
Chern-Simons theory with complex gauge group $\SL(n, \setC)$ discussed
by Witten in \cite{MR1099255}, where he arrives at exactly the formula
\eqref{eq:38}. This relation will be further explored in the final
section of the paper.

We will prove that the curvature of $\tilde \boldnabla_V$, acting on
sections of $\pqb$, has essentially the same expression as the
curvature of $\boldnabla$, acting on sections of $\qb$. We start with
the following.

\begin{proposition}
  \label{prop:12}
  The commutator of the operators $b(V)$ and $\bar b(W)$, acting on
  sections of $\pqb$, is a second-order operator with symbols given by
  \begin{align*}
    \sigma_2\com{b(V)}{\bar b(W)} &= - 4ik \;\! \Theta(V',W'') \\
    \sigma_1\com{b(V)}{\bar b(W)} &= - 4ik \;\! \delta \Theta(V',W'')
    - 8ik \;\! \Theta(V',W'') \trdot dF
    \\
    \sigma_0\com{b(V)}{\bar b(W)} &= - \frac{ik}{4}\;\! \delta \delta
    \Theta (V',W'') - \frac{ik}{4}\;\! r(\Theta(V',W'')) - 2ik\ccint
    \;\!  \theta(V',W'') \\ &\qquad - ik \;\! dF \trdot \Theta(V',W'')
    \trdot dF - ik \;\! \delta \Theta (V',W'') \trdot dF,
    % \frac{\ccint}{2} b(W) V'[F] - \frac{\ccint}{2} b(V) W'[F] =
  \end{align*}
  and furthermore
  \begin{align*}
    &\:\! \sigma_0\com{b(V)}{\bar b(W)} = 2\bar t \, b(V)W''[F] + 2t
    \, \bar b(W) V'[F] - 4ik\ccint \, \theta(V',W''),
  \end{align*}
  for any vector fields $V$ and $W$ on $\mathcal{T}$.
\end{proposition}

\begin{proof}
  The third-order symbol vanishes, as \Fref{prop:3} readily gives
  \begin{align*}
    \sigma_3\com{b(V)}{\bar b(W)} =
    \sigma_3\com{\Delta_{G(V)}}{\Delta_{\bar G(W)}} = 0.
  \end{align*}
  The calculation of the second-order symbol is also straightforward,
  \begin{align*}
    & \sigma_2 \com{b(V)}{\bar b(W)} \\ & \qquad = \sigma_2
    \com{\Delta_{G(V)}}{\Delta_{\bar G(W)}} + 2 \sigma_2
    \com{\nabla^2_{G(V)}}{\nabla_{\bar G(W)\trdot dF}} - 2 \sigma_2
    \com{\nabla^2_{\bar G(W)}}{\Delta_{G(V)\trdot dF}} \\ & \qquad =
    -4ik \Theta(V',W'') + 4\Sym(G(V) \trdot \d \bar \d F \trdot \bar
    G(W)) - 4\Sym(\bar G(W) \trdot \bar \d \d F \trdot G(V)) \\ &
    \qquad = -4ik \Theta(V',W'').
  \end{align*}
  For the first-order symbol, we first calculate
  $\sigma_1\com{\Delta_{G(V)}}{\nabla_{\bar G(W) \trdot dF}}$. We get
  \begin{align*}
    &2\sigma_1\com{\nabla^2_{G(V)}}{\nabla_{\bar G(W) \trdot dF}} \\
    & \qquad = 2\nabla^2_{G(V)} (\bar G(W)\trdot dF) - 4ikG(V) \trdot
    \omega \trdot \bar G(W) \trdot dF + 2G(V)^{xy}R^a_{uxy} \bar
    G(W)^{uv} dF_v
  \end{align*}
  and
  \begin{align*}
    &2\sigma_1\com{\nabla_{\delta G(V)}}{\nabla_{\bar G(W) \trdot dF}} \\
    & \qquad = 2\nabla_{\delta G(V)} (\bar G(W) \trdot dF) - 2 dF
    \trdot \bar G(W) \trdot \nabla \delta G(V) \\ & \qquad =
    2\nabla_{\delta G(V)} (\bar G(W) \trdot dF) + 2idF \trdot \bar
    G(W) \trdot \rho \trdot G(V) - 2dF_u \bar G(W)^{uv} R^a_{vxy}
    G(V)^{xy},
  \end{align*}
  and when combining these, we get
  \begin{align*}
    &2\sigma_1\com{\Delta_{G(V)}}{\nabla_{\bar G(W) \trdot dF}} \\ &
    \qquad = 2\Delta_{G(V)} (\bar G(W) \trdot dF) - 4ikG(V) \trdot
    \omega \trdot \bar G(W) \trdot dF - 2i G(V) \trdot \rho \trdot
    \bar G(W) \trdot dF.
  \end{align*}
  The first term of this expression can be rewritten as
  \begin{align*}
    2\Delta_{G(V)}(\bar G(W) \trdot dF) &= 2\delta (G(V) \trdot \nabla
    (\bar G(W) \trdot dF)) = 2\delta (G(V) \trdot \d \bar \d F \trdot
    \bar G(W)) \\ & = -i \delta (G(V) \trdot \rho \trdot \bar G(W)) +
    i\ccint \delta (G(V) \trdot \omega \trdot \bar G(W)).
  \end{align*}
  For the first-order symbol, we must also understand $\smash{\sigma_1
    \com{\nabla^2_{G(V)}}{W''[F]}}$, which yields
  \begin{align*}
    & 2\sigma_1 \com{\nabla^2_{G(V)}}{W''[F]} = 4 G(V) \trdot dW''[F]
    = - i G(V) \trdot \omega \trdot \delta \bar G(W) - 2i G(V) \trdot
    \omega \trdot \bar G(W) \trdot dF.
  \end{align*}
  Altogether, the previous identities yield
  \begin{align*}
    &2\sigma_1\com{\Delta_{G(V)}}{\nabla_{\bar G(W) \trdot dF}} -
    2\ccint\, \sigma_1\com{\nabla^2_{G(V)}}{W''[F]} \\ & \qquad =
    2\Delta_{G(V)} (\bar G(W) \trdot dF) - 4ikG(V) \trdot \omega
    \trdot \bar G(W) \trdot dF - 2i G(V) \trdot \rho \trdot \bar G(W)
    \trdot dF \\ & \qquad \qquad - i\ccint \delta (\bar G(W) \trdot
    \omega \trdot G(V) ) + 2i\ccint G(V) \trdot \omega \trdot \bar
    G(W) \trdot dF \\ & \qquad = -i \delta (G(V) \trdot \rho \trdot
    \bar G(W)) - 4ikG(V) \trdot \omega \trdot \bar G(W) \trdot dF + 4
    G(V) \trdot \d \bar \d F \trdot \bar G(W) \trdot dF \\ & \qquad
    \qquad + 2i\ccint \, \delta\Theta(V',W'') \\ & \qquad = \delta
    (G(V) \trdot r) \trdot \bar G(W) - 4ikG(V) \trdot \omega \trdot
    \bar G(W) \trdot dF + 4 dF \trdot \bar G(W) \trdot \nabla (G(V)
    \trdot dF) \\ & \qquad \qquad + 2i\ccint \, \delta \Theta(V',W'').
  \end{align*}
  A completely analogous computation shows
  \begin{align*}
    &2\sigma_1 \com{\Delta_{\bar G(W)}}{\nabla_{G(V) \trdot dF}} -
    2\ccint\, \sigma_1 \com{\nabla^2_{\bar G(W)}}{V'[F]} \\ & \qquad =
    \delta (\bar G(W) \trdot r) \trdot G(V) - 4ik \bar G(W) \trdot
    \omega \trdot G(V) \trdot dF + 4 dF \trdot G(V) \trdot \nabla
    (\bar G(W) \trdot dF) \\ & \qquad \qquad + 2i\ccint \, \delta
    \Theta(V',W'').
  \end{align*}
  Finally, combining all of the above with \Fref{prop:3}, we get
  \begin{align*}
    \sigma_1 \com{b(V)}{\bar b(W)} &= \sigma_1
    \com{\Delta_{G(V)}}{\Delta_{\bar G(W)}} + 4\com{G(V) \trdot
      dF}{\bar G(W) \trdot dF} \\ & \qquad +
    2\sigma_1\com{\Delta_{G(V)}}{\nabla_{\bar G(W) \trdot dF}} -
    2\ccint\, \sigma_1\com{\nabla^2_{G(V)}}{W''[F]} \\ & \qquad -
    2\sigma_1 \com{\Delta_{\bar G(W)}}{\nabla_{G(V) \trdot dF}} +
    2\ccint\,
    \sigma_1 \com{\nabla^2_{\bar G(W)}}{V'[F]}\\
    &= \sigma_1 \com{\Delta_{G(V)}}{\Delta_{\bar G(W)}} - 8ik \,
    \Theta(V',W'') \trdot dF \\ & \qquad + \delta (G(V) \trdot r)
    \trdot \bar G(W) - \delta (\bar G(W) \trdot r) \trdot G(V) \\ &= -
    4ik \, \delta \Theta(V',W'') - 8ik \, \Theta(V',W'') \trdot dF
  \end{align*}
  It only remains to calculate the zero-order symbol. First we compute
  \begin{align*}
    2 \sigma_0 \com{\nabla^2_{G(V)}}{\nabla_{\bar G(W) \trdot dF}} &=
    -2ik\,\omega (G(V) \trdot \nabla (dF \trdot \bar G(W))) \\ &=
    -2ik\, \omega (G(V) \trdot \d \bar \d F \trdot \bar G(W)) \\ &=
    -2ik\, \d \bar \d F (G(V) \trdot \omega \trdot \bar G(W)) \\ &= -
    k\:\!\rho (G(V) \trdot \omega \trdot \bar G(W))
    + k\ccint\, \omega (G(V) \trdot \omega \trdot \bar G(W)) \\
    &= - ik\, r( \Theta(V',W'')) - 4ik\ccint \, \theta(V',W''),
  \end{align*}
  and similarly, we have
  \begin{align*}
    2\sigma_0 \com{\nabla^2_{\bar G(W)}}{\nabla_{G(V) \trdot dF}} &=
    ik\, r( \Theta(V',W'') ) + 4ik\ccint \, \theta(V',W'').
  \end{align*}
  Furthermore, we calculate
  \begin{align*}
    &2\sigma_0\com{\nabla_{\delta G(V)}}{\nabla_{\bar G(W) \trdot dF}}
    - 2\sigma_0 \com{\nabla_{\delta \bar G(W)}}{\nabla_{G(V) \trdot
        dF}} \\ & \qquad = -2ik \delta G(V) \trdot \omega \trdot \bar
    G(W) \trdot dF + 2ik \delta \bar G(W) \trdot \omega \trdot G(V)
    \trdot dF \\ & \qquad = -4ik \, \delta \Theta (V',W'') \trdot dF,
  \end{align*}
  and finally
  \begin{align*}
    4\sigma_0 \com{\nabla_{G(V) \trdot dF}}{\nabla_{\bar G(W) \trdot
        dF}} = - 4ik \, dF \trdot \Theta(V',W'') \trdot dF.
  \end{align*}
  All of the above contribute to the zero-order symbol, but we also
  need to calculate $b(V)W''[F]$ and $\bar b(W)V'[F]$. First we
  compute
  \begin{align*}
    &4 \Delta_{G(V)} W''[F] \\ & \qquad = 4 \delta (G(V) \trdot dW''[F]) \\
    &\qquad = - i \delta (G(V) \trdot \omega \trdot \delta \bar G(W))
    - 2i \delta (G(V) \trdot \omega \trdot \bar G(W) \trdot dF) \\
    &\qquad = i \delta \delta (\bar G(W) \trdot \omega \trdot G(V)) -
    2i \d \bar \d F (G(V) \trdot \omega \trdot \bar G(W)) - 2i \delta
    G(V) \trdot \omega \trdot \bar G(W) \trdot dF \\ & \qquad = - i
    \delta \delta(\Theta(V',W'')) - i\:\! r( \Theta(V',W'')) -
    4i\ccint \, \theta(V',W'') - 2i\:\! \delta G(V) \trdot \omega
    \trdot \bar G(W) \trdot dF,
  \end{align*}
  where the last equality used the fact that the double divergence of
  a bivector field only depends on its symmetric part. Similarly, we
  find
  \begin{align*}
    4 \Delta_{\bar G(W)} V'[F] = i \delta \delta\Theta(W'',V') +
    i\:\!r( \Theta(W'',V')) + 4i\ccint \, \theta(W'',V') + 2i \delta
    \bar G(W) \trdot \omega \trdot G(V) \trdot dF.
  \end{align*}
  We also have
  \begin{align*}
    4\nabla_{G(V) \trdot dF} W''[F] &= - i dF \trdot G(V) \trdot
    \omega \trdot \delta \bar G(W) - 2i dF \trdot
    \Theta(V',W'') \trdot dF\\
    4\nabla_{\bar G(W) \trdot dF} V'[F] &= + i dF \trdot \bar G(W)
    \trdot \omega \trdot \delta G(V) + 2i dF \trdot \Theta(W'',V')
    \trdot dF,
  \end{align*}
  so finally, we conclude that
  \begin{align*}
    4b(V)W''[F] &= - i \delta \delta(\Theta(V',W'')) - i\, r(
    \Theta(V',W'')) - 4i\ccint \, \theta(V',W'') \\ & \qquad - 4i
    \delta \Theta(V',W'') \trdot dF - 2i dF \trdot \Theta(V',W'')
    \trdot dF - 8\ccint V'[F]W''[F].
  \end{align*}
  Using this identity, conjugation yields
  \begin{align*}
    \bar b(W) V'[F] = \overline{b(W)V''[F]} = b(V)W''[F],
  \end{align*}
  so the contributions of $\bar b(W) V'[F]$ and $b(V)W''[F]$ to the
  zero-order symbol of the commutator cancel. This finally gives
  \begin{align*}
    \sigma_0 \com{b(V)}{\bar b(W)} &= \sigma_0
    \com{\Delta_{G(V)}}{\Delta_{\bar G(W)}} + 2\sigma_0
    \com{\Delta_{G(V)}}{\nabla_{\bar G(W) \trdot dF}} - 2\sigma_0
    \com{\Delta_{\bar G(W)}}{\nabla_{G(V) \trdot dF}} \\ & \qquad +
    4\sigma_0 \com{\nabla_{G(V) \trdot dF}}{\nabla_{\bar G(W) \trdot
        dF}} - 2\ccint \,b(V)W''[F] + 2\ccint \, \bar b(W) V'[F] \\ &
    = - ik\, \delta \delta \Theta (V',W'') - ik\, r(\Theta(V',W'')) -
    8ik\ccint \, \theta(V',W'') \\ & \qquad - 4ik \, dF \trdot
    \Theta(V',W'') \trdot dF - 4ki \, \delta \Theta (V',W'') \trdot
    dF,
  \end{align*}
  as claimed in the proposition. To prove the final statement of the
  proposition, we note from the above that
  \begin{align*}
    &2 \bar t \, b(V)W''[F] + 2t\, \bar b(W) V'[F] \\ & \qquad = - i k
    \delta \delta(\Theta(V',W'')) - ik\, r ( \Theta(V',W'') ) -
    4ik\ccint \, \theta(V',W'') \\ & \qquad \qquad - 4i k\, dF \trdot
    \Theta(V',W'') \trdot dF - 4i k \, \delta \Theta (V',W'') \trdot
    dF - 4\ccint kV'[F]W''[F],
  \end{align*}
  which gives
  \begin{align*}
    \sigma_0 \com{b(V)}{\bar b(W)} &= 2 \bar t \, b(V)W''[F] + 2t\,
    \bar b(W) V'[F] - 4ik\ccint \, \theta(V',W'') + 4\ccint
    kV'[F]W''[F] \\ & = 2\bar t \sigma_0 \com{b(V)}{W''[F]} + 2t\,
    \sigma_0 \com{\bar b(W)}{V'[F]} - 4ik\ccint \, \theta(V',W'').
  \end{align*}
  This proves the proposition.
\end{proof}

With the bulk of computations encoded in \Fref{prop:12}, we can
calculate the curvature of the Hitchin-Witten connection. This is the
content of the following theorem.

\begin{theorem}
  \label{thm:1}
  The curvature of the Hitchin-Witten connection $\tilde \boldnabla$
  acts as a first-order operator with symbols
  \begin{alignat*}{3}
    \sigma_1 \hcct^{2,0} &= \frac{i}{t^2}
    X'_{\d_{\scale{0.7}{\scriptscriptstyle \mathcal{T}}} \mc } &\qquad
    \quad \sigma_1 \hcct^{1,1} &=0 &\qquad \quad \sigma_1 \hcct^{0,2}
    & = - \frac{i}{\bar t^2} X''_{\bar
      \d_{\scale{0.7}{\scriptscriptstyle \mathcal{T}}} \bar \mc }
    \\[1ex]
    \sigma_0 \hcct^{2,0} &= \frac{t - \ccint}{2t^2}
    \d_{\scriptscriptstyle \mathcal{T}} \mc & \sigma_0 \hcct^{1,1} &=
    \frac{ik\ccint}{t \bar t} \big (\theta - 2i \d_{\scriptscriptstyle
      \mathcal{T}} \bar \d_{\scriptscriptstyle \mathcal{T}} F \big ) &
    \sigma_0 \hcct^{0,2} & = - \frac{\bar t + \ccint}{2\bar t^2} \bar
    \d_{\scriptscriptstyle \mathcal{T}} \bar \mc
  \end{alignat*}
  on sections of the bundle $\pqb$ over $\mathcal{T}$.
\end{theorem}

\begin{proof}
  The calculation of the (2,0)-part proceeds as in the proof of
  \Fref{thm:4} and yields a first-order operator with symbols given by
  \begin{align*}
    \sigma_1 \big (\hcct^{2,0}(V,W) \big ) = \frac{i}{t^2}
    X'_{\d_{\scale{0.7}{\scriptscriptstyle \mathcal{T}}} \mc (V,W) }
    \qquad \text{and} \qquad \sigma_0 \big (\hcct^{2,0}(V,W) \big ) =
    \frac{t - \ccint}{2t^2} \d_{\scriptscriptstyle \mathcal{T}} \mc
    (V,W),
  \end{align*}
  as claimed in the theorem.

  The calculation of the (0,2)-part is completely analogous. We have
  the conjugate statements to \Fref{prop:22},
  \begin{align*}
    \sigma_1\com{\bar b(V)}{\bar b(W)} &= - 4i X''_{\bar
      \d_{\scale{0.7}{\scriptscriptstyle \mathcal{T}}} \bar \mc (V,W)}
    \qquad \text{and} \qquad \sigma_0 \com{\bar b(V)}{\bar b(W)} =
    -2\ccint \bar \d_{\scriptscriptstyle \mathcal{T}} \bar c(V,W),
  \end{align*}
  and \Fref{prop:11},
  \begin{align*}
    d_{\scriptscriptstyle \mathcal{T}} \bar b(V, W) &= i
    \Delta_{\Theta(V,W)} + 2i\nabla_{\Theta(V,W) \trdot dF} - 2
    \nabla_{\bar G(V)\trdot dW[F] - \bar G(W) \trdot dV[F]} - 2 \ccint
    \d_{\scriptscriptstyle \mathcal{T}} \bar \d_{\scriptscriptstyle
      \mathcal{T}} F(V,W),
  \end{align*}
  and these can be combined with the conjugate of \Fref{prop:24} to
  compute
  \begin{align*}
    \sigma_1 \big (\hcct^{0,2}(V,W) \big ) = - \frac{i}{\bar t^2}
    X''_{\bar \d_{\scale{0.7}{\scriptscriptstyle \mathcal{T}}} \bar
      \mc (V,W) } \qquad \text{and} \qquad \sigma_0 \big
    (\hcct^{2,0}(V,W) \big ) = - \frac{\bar t + \ccint}{2\bar t^2}
    \d_{\scriptscriptstyle \mathcal{T}} \mc (V,W).
  \end{align*}
  It only remains to compute the (1,1)-part, and we get
  \begin{align*}
    4t \bar t\, \hcct^{1,1}(V',W'') &= - \com{b(V)}{\bar b(W)} + 2\bar
    t\, d_{\scriptscriptstyle \mathcal{T}} b(V',W'') - 2t\,
    d_{\scriptscriptstyle \mathcal{T}}\bar b(V',W'') \\ & \qquad +
    2\bar t \com{b(V)}{W''[F]} + 2t \com{\bar b(W)}{V'[F]}.
  \end{align*}
  By \Fref{prop:12} and the above expression for
  $d_{\scriptscriptstyle \mathcal{T}} \bar b$, the second-order symbol
  vanishes,
  \begin{align*}
    4 t \bar t \,\sigma_2\,\! \big (\hcct^{1,1}(V',W'') \big ) &= -
    \sigma_2 \com{b(V)}{\bar b(W)} + 2 \bar t \,\sigma_2 \big(
    d_{\scriptscriptstyle \mathcal{T}} b(V',W'') \big ) - 2 t \,
    \sigma_2 \big (d_{\scriptscriptstyle \mathcal{T}} \bar b(V',W'')
    \big ) \\ &= 4ik \, \Theta(V',W'') - 2i (t + \bar t) \;\!
    \Theta(V',W'') \\ &= 0.
  \end{align*}
  The same is true for the first-order symbol
  \begin{align*}
    4 t \bar t \,\sigma_1\,\! \big (\hcct^{1,1}(V',W'') \big ) &= -
    \sigma_1 \com{b(V)}{\bar b(W)} + 2i \bar t \;\! \sigma_1 \big(
    d_{\scriptscriptstyle \mathcal{T}} b(V',W'') \big ) - 2i t \;\!
    \sigma_1 \big (d_{\scriptscriptstyle \mathcal{T}} \bar b(V',W'')
    \big ) \\ & \qquad + 2\bar t \, \sigma_1 \com{b(V)}{W''[F]} + 2t\,
    \sigma_1 \com{\bar b(W)}{V'[F]} \\ & = 4ik\;\! \delta
    \Theta(V',W'') + 8ik\;\! \Theta(V',W'') \trdot dF \\ & \qquad -
    2i(t + \bar t) \delta \Theta(V',W'') - 4i(t + \bar t)
    \Theta(V',W'') \trdot dF \\ & \qquad - 4\bar t G(V) \trdot W''[F]
    - 4t \bar G(W) \trdot dV'[F] \\ & \qquad + 4\bar t G(V) \trdot
    W''[F] + 4t \bar G(W) \trdot dV'[F] \\ &= 0.
  \end{align*}
  Finally, the last statement of \Fref{prop:12} gives the zeroth-order
  part
  \begin{align*}
    4 t \bar t \,\sigma_0 \big ( \hcct^{1,1}(V',W'') \big) &=
    4ik\ccint \, \theta(V',W'') + 4\ccint (\bar t + t)
    \d_{\scriptscriptstyle \mathcal{T}} \bar \d_{\scriptscriptstyle
      \mathcal{T}} F(V',W'') \\ & = 4ik\ccint \big ( \theta(V',W'') -
    2i \d_{\scriptscriptstyle \mathcal{T}} \bar \d_{\scriptscriptstyle
      \mathcal{T}} F(V',W'') \big ).
  \end{align*}
  This completes the proof of the theorem.
\end{proof}

\subsection*{Projective Flatness}

In case the Hitchin connection $\boldnabla$ in the bundle $\qb$ over
$\mathcal{T}$ is projectively flat, the parallel translation maps
along homotopic curves are equal up to scale. Thus, if the parameter
space $\mathcal{T}$ is simply connected, the Hitchin connection gives
a canonical identification of the projectivized quantum spaces
associated with different complex structures. In this sense, the
quantization is independent of the complex structure.

The following theorem is an immediate consequence of the explicit
curvature calculations of \Fref{thm:4} and \Fref{thm:1}.

\begin{theorem}
  \label{thm:2}
  The connections $\boldnabla$ and $\tilde \boldnabla$, acting on
  $\qb$ and $\pqb$, respectively, are both projectively flat if and
  only if the holomorphic vector field on $M$ given by
  \begin{align*}
    X'_{\d_{\scale{0.7}{\scriptscriptstyle \mathcal{T}}} \mc (V,W)} =
    \frac{i}{4} \Delta_{G(W)} \delta G(V) - \frac{i}{4} \Delta_{G(V)}
    \delta G(W) + \frac{i}{2} G(W) \trdot d \mc(V) - \frac{i}{2} G(V)
    \trdot d \mc(W)
  \end{align*}
  vanishes, for all vector fields $V$ on $\mathcal{T}$.
\end{theorem}

\begin{proof}
  Projective flatness amounts to the curvature being a two-form on
  $\mathcal{T}$ with values in constant functions on $M$. In
  particular, it takes values in zeroth-order differential operators.
  Being proportional to the first-order symbol of the (2,0)-part of
  the curvature of both $\boldnabla$ and $\tilde \boldnabla$, the
  vector field $X'_{\d_{\scale{0.7}{\scriptscriptstyle \mathcal{T}}}
    \mc (V,W)}$ will of course vanish if either of these are
  projectively flat.

  Conversely, suppose that $\smash{X'_{\d_{\scriptscriptstyle
        \mathcal{T}} \mc (V,W)}}$ vanishes, which means that the
  curvature of $\boldnabla$ is a zeroth-order operator, and since it
  preserves the subbundle $\qb$ of holomorphic sections of
  $\mathcal{L}^k$ it takes values in holomorphic functions on $M$. By
  assumption, such functions are constants, which proves that
  $\boldnabla$ is projectively flat. In particular, the form
  $\d_{\scriptscriptstyle \mathcal{T}} \mc$ takes values in constant
  functions on $M$, which is also the case for $\theta - 2i
  \d_{\scriptscriptstyle \mathcal{T}} \bar \d_{\scriptscriptstyle
    \mathcal{T}} F$, as already noted in \Fref{prop:1}. By
  \Fref{thm:1}, this proves that the Hitchin-Witten connection $\tilde
  \boldnabla$ is projectively flat.
\end{proof}

We observe that our main \Fref{thm:MT} follows as an immediate
corollary of the theorem above. This ends the general discussion of
the Hitchin and the Hitchin-Witten connection. In the final section,
we will apply the results to the setting originally motivating the
study of these connections.

\chapter{Quantum Chern-Simons Theory}

\label{QCST}

Is this section, we apply the results of previous sections to
Chern-Simons theory with either compact or complex gauge group. To be
specific, let $G$ denote the real Lie group $\SU(n)$, sitting as the
maximal compact subgroup of its complexification $G_\sssetC$, which
can be identified with $\SL(n,\setC)$. Indeed, if $\g$ denotes the Lie
algebra of $G$, consisting of skew-Hermitian traceless matrices, then
its complexification $\gc$ can be identified with the Lie algebra
$\mathfrak{sl}_n(\setC)$ of traceless matrices through the unique
splitting of any complex matrix into Hermitian and skew-Hermitian
parts. Furthermore, let $\langle \cdot, \cdot \rangle$ be an invariant
inner product on $\gc$, such as $-\frac{1}{8\pi^2}\!\Tr$, which is
normalized so that $\frac{1}{6}\langle \vartheta \wedge [ \vartheta
\wedge \vartheta] \rangle$ represents an integral generator of $H^3(G,
\setR)$, where $\vartheta$ is the Maurer-Cartan form on $G$.

Let $\Sigma$ be a closed surface of genus $g \geq 2$, and let
$\Sigma_p$ be the surface obtained by puncturing $\bar \Sigma$ at a
point $p \in \Sigma$. Fix an element $d \in \setZ/n\setZ$ and a small
loop $\gamma$ around the puncture, and consider the moduli spaces
\begin{align}
  \label{eq:26}
  \begin{aligned}
    M &=\Hom_d(\pi_1(\Sigma_p), G) / G \\
    M_{\sssetC} &= \Hom^+_d(\pi_1(\Sigma_p), G_{\sssetC} ) /
    G_{\sssetC},
  \end{aligned}
\end{align}
of representations of the fundamental group mapping $\gamma$ to the
element $e^{2\pi i d/n} I$ in the common center of $G$ and
$G_\sssetC$. For $M_\sssetC$, we restrict to reductive
representations, which is indicated by a plus in the notation. The
subspaces $M^s$ and $M^s_\sssetC$ corresponding to irreducible
connections are of particular relevance to us. For these spaces, we
have en embedding $M^s \subset M^s_\sssetC$, since any irreducible
unitary representation can only be conjugate to another such by a
unitary transformation.

The spaces $M$ and $M_\sssetC$ also have a gauge theoretic realization
as moduli spaces of flat connections. Let $P$ denote the trivial
principal bundle over $\Sigma_p$ with structure group $G$. Fix an
element $a_d \in \g$ such that $\exp(2\pi a_d) = e^{2\pi i d / n} I$,
and denote by $\mathcal{F}$ the space of flat connections on $P$ which
are all equal to $a_d d\theta$ for some fixed polar coordinate system,
with angular coordinate $\theta$, in a disc around the puncture. These
connections are said to be in temporal gauge near the puncture and
clearly have holonomy $e^{2\pi i t/n}\Id$ around $\gamma$. If
$\mathcal{G}$ denotes the space of gauge transformations equal to the
identity in a neighbourhood of the puncture, then the holonomy
representations provide an identification
\begin{align*}
  M \cong \mathcal{F}/\mathcal{G}.
\end{align*}
Similarly, let $\mathcal{F}^+_\sssetC$ be the space of flat reductive
connections on the trivial $G_{\sssetC}$-bundle $P_\sssetC$ over
$\Sigma_p$, and let $\mathcal{G}_\sssetC$ be the space of gauge
transformations, both spaces with restrictions similar to the above
around the puncture. The holonomy representation then gives an
identification
\begin{align*}
  M_{\sssetC} \cong \mathcal{F}^+_{\sssetC} / \mathcal{G}_\sssetC.
\end{align*}
By a more careful selection of connections and gauge transformations,
using exponential decay in weighted Sobolev norms around the puncture
\cite{MR1605216}, the subsets $M^s$ and $M^s_\sssetC$ corresponding to
irreducible connections can be endowed with the structure of smooth
manifolds. The tangent space $T_{[A]}M^s_{\sssetC}$ at a connection
$A$, is given by the compactly supported first cohomology
$H^1_A(\Sigma_p, \gc)$, with values in the adjoint bundle and exterior
derivative $d_A$ induced by $A$. We can then define a complex
symplectic form $\omega_\sssetC$ by
\begin{align}
  \label{eq:25}
  \omega_\sssetC ([\alpha], [\beta]) = - 4\pi \int_{\Sigma} \langle
  \alpha \wedge \beta \rangle,
  % = \frac{1}{2\pi} \int_{\Sigma} \Tr(\alpha \wedge \beta),
\end{align}
In complete analogy, the tangent space $T_{[A]}M^s$ is given by the
cohomology $H^1_A(\Sigma_p, \g)$. The formula \eqref{eq:25} defines in
this case a real symplectic form on $M^s \subset M^s_\sssetC$, which
is of course just the restriction of $\omega_\sssetC$.

If $t \in \setC$ is a complex number with integer real part, $k = \re
t \in \setZ$, we wish to quantize the space $M^s_\sssetC$ with respect
to the real symplectic form
\begin{align*}
  \omega_t = \frac{1}{2} ({t\omega_\sssetC + \overline{t
      \omega_\sssetC}}),
\end{align*}
as well as the space $M^s$ equipped with the restriction of
$\omega_t$, which is given by $k\omega$.

\pagebreak[1] Suppose that $A \in \mathcal{F}^+_\sssetC$ is a
connection, and that $g \colon \Sigma_p \to G_\sssetC$ is a gauge
transformation in $\mathcal{G}_\sssetC$. Since $G_\sssetC$ is simply
connected and hence 2-connected, we can choose a homotopy $\tilde g
\colon \Sigma_p \times [0,1] \to G_\sssetC$ from the trivial gauge
transformation to $g$, keeping identity values fixed. If $\pi \colon
\Sigma_p \times [0,1] \to \Sigma_p$ denotes projection onto the first
factor, we can pull back $A$ to a connection $\tilde A = \pi^* A$ on
$\Sigma_p \times [0,1]$ and consider the Chern-Simons form
\begin{align*}
  \alpha_\sssetC(\tilde A) = \langle \tilde A \wedge F_{\tilde A}
  \rangle - \frac{1}{6} \langle \tilde A \wedge [\tilde A \wedge
  \tilde A] \rangle,
  % = - \frac{1}{8\pi^2}\Tr \big(\tilde A \wedge
  % F_{\tilde A} - \frac{1}{6} \tilde A \wedge [\tilde A \wedge \tilde
  % A] \big)
\end{align*}
which will of course be complex-valued in general. Using the complex
number $t \in \setC$ from before, we shall consider the real form
\begin{align*}
  \alpha_t(\tilde A) = \frac{1}{2}(t \alpha_\sssetC (\tilde A) +
  \overline{t\alpha_\sssetC(\tilde A)}),
\end{align*}
and the Chern-Simons cocycle given by
\begin{align*}
  \Theta_t(A, g) = \exp \Big ( 2\pi i
  \int_{\smash{\raisebox{-3pt}{$\mathrlap{\scriptstyle \Sigma \times
          [0,1]}$}}} \;\; \alpha_t (\tilde A^{\tilde g}) \Big ) \:\:
  \in \: \U(1),
\end{align*}
where $\tilde A^{\tilde g}$ denotes the gauge-transformed connection.
Since $A$ is flat and in temporal gauge, the Chern-Simons form
vanishes in a neighbourhood of $\{p\} \times [0,1]$ and the integral
converges. Furthermore, since the real part of $t$ is an integer, the
expression is independent of the choice of homotopy $\tilde g$ and
defines a map $\Theta \colon \mathcal{F_\sssetC} \times
\mathcal{G}_\sssetC \to U(1)$, which can be shown to satisfy the
cocycle relation
\begin{align*}
  \Theta_t(A^g, h) \Theta_t(A, g) = \Theta_t (A, gh).
\end{align*}
This cocycle can be used to lift the action by the gauge group
$\mathcal{G}_\sssetC$ on $\mathcal{F}^+_\sssetC$ to the trivial line
bundle $\mathcal{F}^+_\sssetC \times \setC$ by
\begin{align*}
  (A, z) \cdot g = (A^g, \Theta_t(A, g) z),
\end{align*}
and the quotient defines a Hermitian line bundle,
\begin{align*}
  \mathcal{L}^t_\sssetC \to M^s_\sssetC,
\end{align*}
over the smooth part of the moduli space. Furthermore, this bundle
comes with a unitary connection $\nabla$, given on the trivial bundle
$\mathcal{F}^+_\sssetC \times \setC$ by the one-form
\begin{align*}
  B_t(\alpha) = 2\pi i\int_\Sigma \langle tA \wedge \alpha +
  \overline{ tA \wedge \alpha}\rangle,
  % = - \frac{i}{4\pi} \int_\Sigma \Tr (tA \wedge \alpha +
  % \overline{ tA \wedge \alpha}),
\end{align*}
for a tangent vector $\alpha \in T_A\mathcal{F} \cong \Omega^1(\Sigma,
\gc)$ at $A \in \mathcal{F}^+_\sssetC$. It is easily verified that the
curvature of this connection is given by
\begin{align*}
  F_\nabla = -i\omega_t,
\end{align*}
so that $\mathcal{L}^t_\sssetC$ defines a prequantum line bundle over
$M^s_\sssetC$, with symplectic form $\omega_t$, and clearly this
restricts to a prequantum line bundle $\mathcal{L}^k$ over $M_s$,
equipped with $k\omega$ as its symplectic structure.

To perform geometric quantization, we must introduce a
polarization. This will come from a choice of Riemann surface
structure on $\Sigma$, which amounts to a Hodge star-operator $*
\colon \Omega^1(\Sigma) \to \Omega^1(\Sigma)$, satisfying $*^2 =
{-\Id}$ for dimensional reasons. Using Hodge theory to identify
$H^1_A(\Sigma, \gc)$ with the space of harmonic forms, we define an
almost complex structure $J$ on $M^s_\sssetC$ by the expression
\begin{align}
  \label{eq:4}
  J \alpha = {-{*\overline{\alpha}}}
\end{align}
on harmonic representatives, the space of which is preserved by
$*$. It is easily checked that $J$ is compatible with
$\omega_t$. There is another obvious almost complex structure $I$ on
$M^s_\sssetC$, which is simply given by $I\alpha = i\alpha$, and
clearly this anti-commutes with $J$. Both of these almost complex
structure are in fact integrable, as seen through the correspondence
with Higgs bundles discussed below, giving $M^s_\sssetC$ the structure
of a hyperk\"ahler manifold.

The K\"ahler structure $J$ only depends on the Riemann surface
structure up to isotopy, so in fact we get a family $J \colon
\mathcal{T}_\Sigma \to C^\infty(M^s_\sssetC, \End(TM^s_\sssetC))$ of
K\"ahler structures parametrized by the Teichm\"uller space
$\mathcal{T}_{\Sigma}$ of the surface. At a unitary connection $A$,
the family $J$ clearly preserves the real subspace $H^1_A(\Sigma, \g)$
tangent to $M^s$, so $J$ also defines a family of K\"ahler structures
on $M^s$. This is indeed the polarization we shall use for quantizing
$M^s$, but for $M^s_\sssetC$ we shall instead rely on a certain real
polarization, considered by Witten in \cite{MR1099255}, which also
depends on the Riemann surface structure of $\Sigma$. Let us first
recall how the quantization of the moduli space $M^s$ for compact
gauge group proceeds.

The moduli space $M^s$ is simply connected, so in particular $H^1(M^s,
\setR)$ vanishes, and furthermore $H^2(M^s, \setZ) = \setZ$ (see
\cite{MR702806,MR1605216,1408.2499}). For a given point $\sigma \in
\mathcal{T}_\Sigma$ in Teichm\"uller space, the properties of
$M^s_\sigma$ as a complex manifold can be understood through the
classical work of Narasimhan and Seshadri \cite{MR0166799}, which
identifies $M$ with the moduli space of S-equivalence classes of
semi-stable holomorpic vector bundles, over the Riemann surface
$\Sigma$, of rank $n$, degree $d$ and fixed determinant.  This space
has the structure of an normal projective algebraic variety which is
typically singular but contains the moduli space of stable bundles as
an open smooth subvariety, corresponding exactly to the manifold $M^s$
of irreducible connections. In this picture, the holomorphic tangent
space $T_{[E]}M^s$ at a stable holomorphic bundle $E$ is given by the
cohomology $H^1(\Sigma, \End_0 \!E)$, with values in the traceless
endomorphisms. As we saw above, the moduli space $M^s$ admits a
prequantum line bundle $\mathcal{L}$, which is just the determinant
line bundle in this picture. It generates the Picard group, so in
particular $[\tfrac{\omega_1}{2\pi}]$ is a generator of $H^2(M^s,
\setZ)$, and $c_1(M^s) = \ccint [\frac{\omega_1}{2\pi}]$ with $\ccint
= 2\gcd(r,d)$ as proved in \cite{MR999313}.

If the rank $n$ and the degree $d$ are coprime, there are no strictly
semi-stable bundles, so $M = M^s$ is a compact K\"ahler manifold. In
this case, clearly $H^0(M^s, \mathcal{O}) = \setC$, but as observed by
Hitchin \cite{MR1065677}, this holds even for non-coprime $n$ and $d$
by the Hartogs theorem, since the complement of $M^s$ in $M$ has
codimension at least 2. Hitchin also notes that the family $J$ of
K\"ahler structures on $M^s$ parametrized by Teichm\"uller space,
which is itself in a canonical way a contractible complex manifold, is
holomorphic and rigid in the sense of \Fref{def:3} and \Fref{def:4}.
Finally, the K\"ahler metric on the moduli space of stable bundles was
studied by Zograf and Takhtajan in \cite{MR1018746}, where they give a
Ricci potential in terms of the determinant of the Laplacian on the
endomorphism bundle of a stable bundle.

Altogether, the discussion above demonstrates that the
moduli space $M^s$ satisfies all the conditions of \Fref{thm:58} to
ensure the existence of a Hitchin connection. Furthermore, in the case
of coprime $n$ and $d$, Narasimhan and Ramanan \cite{MR0384797} have
shown that $M^s$ does not admit any holomorphic vector fields, so
\Fref{thm:MT} implies that the Hitchin connection must be projectively
flat.

Projective flatness also holds in the non-coprime case, which
similarly does not admit holomorphic vector fields. Hitchin
\cite{MR1065677} proves this by regarding a holomorphic vector field
on $M^s$ as a holomorphic function on the cotangent bundle $T^*M^s$,
which sits inside the moduli space of semi-stable Higgs bundles
$\mathcal{M}$, as discussed below. Once again appealing to the Hartogs
theorem, the function, which is homogeneous of degree 1 in the action
of $\setC^*$ and constant along fibers of the Hitchin fibration, can
be extended to $\mathcal{M}$, which is quite easily seen not to
support such functions. The proof does not apply to the special
situation when the genus $g$ and the rank $n$ are both equal to 2, and
the holonomy around $\gamma$ is trivial, because the moduli space
$M_\sssetC$ is isomorphic to $\setC P^3$ and clearly has holomorphic
vector fields.

In summary, we have proved projective flatness of the Hitchin
connection for the moduli spaces as claimed in \Fref{thm:MTCS}. Our
proof of projective flatness follows the original proof by Hitchin
\cite{MR1065677} from the point where the curvature is known to be at
most order one, but vanishing of the higher-order terms was
established by Hitchin through particular properties of the Hitchin
integrable system \cite{MR885778}, whereas we derive it from rigidity
of the family of K\"ahler structures.

We now turn to the problem of quantizing the moduli space
$M^s_\sssetC$ of irreducible $G_\sssetC$-connections. The theory of
Higgs bundles naturally enters in this discussion as well. Suppose
again that $\Sigma$ is endowed with a Riemann surface structure and
recall that a Higgs bundle is a pair $(E, \Phi)$, where $E \to \Sigma$
is a holomorphic vector bundle and the Higgs field $\Phi \in
H^0(\Sigma, \End_0\! E \otimes K)$ is a holomorphic one-form with
values in the traceless endomorphisms of $E$. Stability for Higgs
bundles is defined by imposing the usual stability condition on the
slope of subbundles $E' \subset E$, but only the invariant ones
satisfying $\Phi(E') \subset E' \otimes K$.  This leads to the moduli
spaces $\mathcal{M}$ and $\mathcal{M}^s$, of semi-stable, respectively
stable, Higgs bundles of rank $n$, degree $d$ and fixed
determinant. Nitsure \cite{MR1085642} proves that $\mathcal{M}$ is a
quasi-projective algebraic variety, which contains $\mathcal{M}^s$ as
an open smooth subvariety. Through the work of Hitchin, Simpson,
Donaldson and Corlette
\cite{MR887284,MR944577,MR1179076,MR887285,MR965220}, the moduli space
$\mathcal{M}$ can be identified, via the Hitchin equations and
non-abelian Hodge theory, with the moduli space $M_\sssetC$ of flat
reductive $G_\sssetC$-connections, with stable bundles corresponding
to irreducible connections.

If $E$ is itself already a stable bundle, then obviously stability as
a Higgs bundle is implied for any Higgs field $\Phi$. In particular,
by simply taking the Higgs field to be zero, the moduli space $M^s$
sits canonically inside $\mathcal{M}^s \cong M^s_\sssetC$, as we
discussed already at the level of representations. But in fact, the
entire cotangent bundle of $M^s$ can be embedded in $M^s_\sssetC$,
although the embedding depends crucially on the Riemann surface
structure of $\Sigma$. Indeed, we get the following identification by
Serre duality,
\begin{align*}
  H^0(\Sigma, \End_0\! E \otimes K) \cong H^1(\Sigma, \End_0 \!E)^*,
\end{align*}
where we recognize the right-hand side as the holomorphic cotangent space of the
moduli space $M^s$. In other words, the Higgs fields on stable bundles can
be viewed as cotangent vectors to $M^s$ through Serre duality.
The Higgs bundle model $\mathcal{M}^s$ of $M_\sssetC^s$ illuminates the
hyperk\"ahler structure but it also carries a natural action of
$\setC^*$ by scaling the Higgs field.

To quantize the modulu space $M^s_\sssetC$, we follow Witten
\cite{MR1099255} and define a family of real polarizations,
parametrized by the Teichm\"uller space of $\Sigma$, in the following
way. For any Riemann surface structure on $\Sigma$, in the form of a
star-operator $*$ as usual, we can use Hodge theory to split the space
$H_A^1(\Sigma, \gc)$ into types,
\begin{align*}
  H_A^1(\Sigma, \gc) = H^{1,0}_A(\Sigma, \gc) \oplus H^{0,1}_A(\Sigma,
  \gc).
\end{align*}
Although this uses the complex structure on $H_A^1(\Sigma, \gc)$, it
defines a splitting of the underlying real space, which 
is a model for the real tangent space $T_{[A]}M^s_\sssetC$. Since the
symplectic form $\omega_\sssetC$ is invariant under the Hodge-star, the
same of course applies to $\omega_t$, so that each of the summands are
Lagrangian. 

We shall take the subspaces $H^{1,0}_A(\Sigma, \gc)$ as a
polarization, and define the quantum space to be the polarized
sections of $\mathcal{L}_\sssetC$, which are covariantly constant along
its leaves. Notice that the polarization at a unitary connection $A$
is transverse to the tangent space $T_{[A]}M^s = H^1_A(\Sigma, \g)$,
simply because the only real form of type (1,0) is the zero form. This
means that a polarized section is determined by its values on the
space $M^s$, or in other words, that the quantum space is identified
with the prequantum space $C^\infty(M^s, \mathcal{L}^k)$ of smooth
sections over $M^s$, although the identification depends on the
complex structure on $\Sigma$. Surely this identification requires
covariantly constant sections to exist on the leaves of the foliation,
which would be guaranteed for instance if they were simply connected,
but we shall not deal with this question here. Instead, the discussion
can be taken as motivation for using $C^\infty(M^s, \mathcal{L}^k)$ as
a model for the quantum space of $M^s_\sssetC$.

To understand how the identification between the quantum space of
$M^s_\sssetC$ and the prequantum space $C^\infty(M^s, \mathcal{L}^k)$ depends
on the Riemann surface structure, we consider the latter as the fiber
of a trivial bundle over Teichm\"uller space,
\begin{align*}
  \mathcal{T}_\Sigma \times C^\infty(M^s, \mathcal{L}^k) \to
  \mathcal{T}_\Sigma.
\end{align*}
Then the expression \eqref{eq:38} defines a connection $\tilde
\boldnabla_V$ on this bundle, which is projectively flat by
\Fref{thm:MT}, once again due to the fact that the moduli space $M^s$
does not admit any holomorphic vector fields for any of the K\"ahler
structures in the family $J$ parametrized by Teichm\"uller space. This
establishes the second statement of \Fref{thm:MTCS} on the
Hitchin-Witten connection. We also expand the statement in the
following theorem.

\begin{theorem}
  \label{thm:3}
  For any $t \in \setC$ with $k = \re(t) \in \setZ$, the trivial
  bundle $\mathcal{T}_\Sigma \times C^\infty(M^s, \mathcal{L}^k)$ over
  Teichm\"uller space, with fiber given by the smooth sections of the
  Chern-Simons line bundle over the moduli space of flat $\SU(n)$
  connections, has a projectively flat connection given by
  \begin{align*}
    \tilde \boldnabla_V \!=\! \trivcon_V \! + \!\frac{1}{2t}
    (\Delta_{G(V)} \! + \!2 \nabla_{\!G(V) \trdot dF}\! -\! 2\ccint
    V'[F]) - \frac{1}{2\bar t} (\Delta_{\bar G(V)} \! + \!2 \nabla_{\!
      \bar G(V) \trdot dF} \!- \! 2\ccint V''[F]) \!+ \! V[F],
  \end{align*}
  for any vector field $V$ on Teichm\"uller space.
\end{theorem}
In this way, the quantum spaces arising from different real
polarizations as above are projectively identified through the
parallel transport of the Hitchin-Witten connection.

%\fxnote{switch AU3 and AU4}

\clearpage

\nocite{MR2330673,MR1865275}

%%% Change bibliography to references
\renewcommand{\bibname}{References}

% \bibliographystyle{/Users/nlg/Dropbox/Math/Latex/Bibtexstyles/abbrvalphanumsortplain}  
% \bibliography{/Users/nlg/Dropbox/Math/Latex/Bibliography/FullBibliography}

%\bibliographystyle{abbrvalphanum}
%\bibliography{FullBibliography}

\newcommand{\etalchar}[1]{$^{#1}$}
\def\cprime{$'$} \def\cprime{$'$} \def\cprime{$'$}

\end{document}